\documentclass[a4paper,11pt,reqno]{amsart}
\usepackage[left=2cm,right=2cm, top=3cm, bottom=3cm]{geometry}
\usepackage[usenames, dvipsnames]{xcolor}
\usepackage{wasysym}
\usepackage{amsmath,amsthm,amssymb, bm, mathrsfs}
\usepackage{tikz-cd}
\usepackage[all]{xy}
\usepackage{mathtools, enumitem}
\setenumerate{label=(\roman*), font=\upshape, labelindent=0pt, itemindent=0pt, leftmargin=*}
\usepackage[style=numeric-comp, backend=bibtex, doi=false, url=false, firstinits=true, eprint=false, sorting=nyt, maxnames=6]{biblatex}
\DeclareFieldFormat[article, inbook, incollection, inproceedings, thesis, unpublished, misc]{title}{\mkbibemph{#1}\addperiod}
\DeclareFieldFormat{journaltitle}{#1,}
\renewbibmacro{in:}{}
\DeclareFieldFormat[article]{volume}{ {#1}} 
\DeclareFieldFormat[article, inbook, incollection, inproceedings, misc, thesis, unpublished]{number}{\mkbibbold{\mkbibparens{#1}}} 
\DeclareFieldFormat[article]{pages}{#1}
\renewbibmacro*{volume+number+eid}{
  \printfield{volume}
  \printfield{number}
  \addcomma
}
\renewbibmacro*{issue+date}{%
  \setunit{\addcomma\space}
    \iffieldundef{issue}
      {\usebibmacro{date}}
      {\printfield{issue}%
       \setunit*{\addspace}%
       \usebibmacro{date}}
  \newunit}

\makeatletter
\DeclareFontFamily{OMX}{MnSymbolE}{}
\DeclareSymbolFont{MnLargeSymbols}{OMX}{MnSymbolE}{m}{n}
\SetSymbolFont{MnLargeSymbols}{bold}{OMX}{MnSymbolE}{b}{n}
\DeclareFontShape{OMX}{MnSymbolE}{m}{n}{
    <-6>  MnSymbolE5
   <6-7>  MnSymbolE6
   <7-8>  MnSymbolE7
   <8-9>  MnSymbolE8
   <9-10> MnSymbolE9
  <10-12> MnSymbolE10
  <12->   MnSymbolE12
}{}
\DeclareFontShape{OMX}{MnSymbolE}{b}{n}{
    <-6>  MnSymbolE-Bold5
   <6-7>  MnSymbolE-Bold6
   <7-8>  MnSymbolE-Bold7
   <8-9>  MnSymbolE-Bold8
   <9-10> MnSymbolE-Bold9
  <10-12> MnSymbolE-Bold10
  <12->   MnSymbolE-Bold12
}{}
\let\llangle\@undefined
\let\rrangle\@undefined
\DeclareMathDelimiter{\llangle}{\mathopen}%
                     {MnLargeSymbols}{'164}{MnLargeSymbols}{'164}
\DeclareMathDelimiter{\rrangle}{\mathclose}%
                     {MnLargeSymbols}{'171}{MnLargeSymbols}{'171}
\makeatother
\newcommand{\bil}[2]{\left\llangle #1, #2 \right\rrangle}

\newcommand{\R}{\mathbb{R}} 
\newcommand{\Q}{\mathbb{Q}} 
\newcommand{\uhp}{\mathbb{H}} 
\newcommand{\C}{\mathbb{C}} 
\newcommand{\Z}{\mathbb{Z}} 
\newcommand{\N}{\mathbb{N}} 

\renewcommand{\Im}{\mathrm{Im}}
\renewcommand{\Re}{\mathrm{Re}}
\newcommand{\e}{\mathfrak{e}}
\newcommand{\abs}[1]{\left|#1\right|}
\newcommand{\K}{\mathcal{K}}

\newcommand{\legendre}[2]{\left( \frac{\mathstrut #1}{#2} \right)} 

\DeclareMathOperator{\SL}{SL}
\DeclareMathOperator{\sgn}{sgn}

\DeclareMathOperator{\PSL}{PSL}
\DeclareMathOperator{\Mp}{Mp}

\DeclareMathOperator{\tr}{tr}

\renewcommand{\H}{\uhp}

\newtheorem{theorem}{Theorem}
\newtheorem{proposition}[theorem]{Proposition}
\newtheorem{lemma}[theorem]{Lemma}

\numberwithin{theorem}{section}
\numberwithin{equation}{section}

\theoremstyle{definition}
\newtheorem{example}[theorem]{Example}
\newtheorem{definition}[theorem]{Definition}
\newtheorem{remark}[theorem]{Remark}

\usepackage{todonotes}
\newcommand{\Amod}[1]{A_{#1}^{\mathtt{mod}}}

\newcommand{\Aexp}[1]{A_{#1}^!}

\newcommand{\calQ}{\mathcal{Q}}

\newcommand*{\bigs}[1]{\vcenter{\hbox{\scalebox{1.5}{\ensuremath#1}}}}

\DeclareMathOperator*{\CT}{CT}
\DeclareMathOperator*{\erfc}{erfc}

\usepackage{placeins}
\usepackage[most, breakable]{tcolorbox}
\usepackage{multicol}
\usepackage{bbm}

\usepackage[pdfborder={0 0 0}, colorlinks = true, linkcolor=blue, citecolor=OliveGreen, pdftex]{hyperref}

\renewcommand{\pmod}[1]{\  \,  \left(\mathrm{mod} \,  #1 \right)}

\title{On the modular completion of certain generating functions}
\author{Kathrin Bringmann}
\email{kbringma@math.uni-koeln.de}
\author{Stephan Ehlen}
\email{stephan.ehlen@math.uni-koeln.de}
\author{Markus Schwagenscheidt}
\email{mschwage@math.uni-koeln.de}

\address{Mathematisches Institut, University of Cologne, Weyertal 86-90, D-50931 Cologne, Germany}

\begin{document}
\begin{abstract}
 We complete several generating functions to non-holomorphic modular forms in two variables. For instance, we consider the generating function of a natural family of meromorphic modular forms of weight two. We then show that this generating series can be completed to a smooth, non-holomorphic modular form of weights $\frac32$ and two.
Moreover, it turns out that 
the same function is also a modular completion of the generating function
of weakly holomorphic modular forms of weight $\frac32$, 
which prominently appear in work of Zagier \cite{zagiertraces} on traces of singular moduli.
\end{abstract}
\maketitle

\section{Introduction and statement of results}
\label{sec:introduction}

\subsection{Modular completions}
When studying an interesting sequence $(a_n)_{n\in\Z}$,
it is often helpful to consider the generating function
\[
   \sum_{n\in \Z}a_nq^n.
\]
An important class of examples is given by theta functions associated with positive definite quadratic forms,
which are generating functions of representation numbers.
Studying the analytic properties of such a generating function
provides rich analytic tools to obtain information about the sequence $a_n$.
Famous examples include explicit formulas for the number of representations of a positive integer as a sum of four and eight
squares, whose generating functions are modular forms of weight two and four, respectively,
or the partition function $p(n)$, whose generating function
is essentially a modular form of weight $-\frac{1}{2}$. The latter fact plays a crucial role in the ingenious proof of
Rademacher's exact formula \cite[Theorem~5.1]{andrewspartitions} for $p(n)$.

In a different direction, the following generating function of certain cusp forms $f_{k,d}$ was considered by Kohnen \cite{kohnen85}.
For integers $k \geq 2$, he defined ($e(x):=e^{2\pi ix}$, $\tau,z \in \H := \{z \in \C\ : \Im(z)>0 \}$)
\[
  \Omega_{k}(\tau,z):=\sum_{d= 1}^\infty d^{k-\frac12}f_{k,d}(z)e(d\tau),
\]
where for $d\in\N$
\begin{equation*}
  f_{k,d}(z):=\sum_{Q\in\mathcal Q_{d}}Q(z,1)^{-k}.
\end{equation*}
Here, for a discriminant $\delta$, $\mathcal Q_\delta$ denotes the set of integral binary quadratic forms of discriminant $\delta$. It is not hard to see that
each $f_{k,d}$ is a cusp form of weight $2k$ for the full modular group $\Gamma:=\SL_2(\Z)$.
These functions were introduced by Zagier \cite{zagierrealquadratic} in his study of the Doi-Naganuma lift. Katok \cite{katok} showed that they can be written as hyperbolic Poincar\'e series. Using the modularity of the $f_{k,d}$'s it follows that $z\mapsto \Omega_k(z;\tau)$ is modular of weight $2k$. It turns out that $\tau\mapsto \Omega_k(z;\tau)$ is also modular (of weight $k+\frac12$). To see this, one rewrites $\Omega_k$ as (up to a constant)
\[
  \sum_{n=1}^\infty n^{k-1} \sum_{d\mid n} d^k P_{k+\frac12,d^{2}}(\tau) e(nz),
\]
where the functions $P_{k+\frac12,d^{2}}$ are certain exponential Poincar\'e series. A key property of $\Omega_k(\tau,z)$ is that it is the holomorphic kernel function for the Shimura and Shintani lifts. To be more precise, for $f$ a cusp form of weight $2k$ and $g$ a cusp form of weight $k+\frac12$ in Kohnen's plus-space, the Shimura lift of $g$ basically equals $\langle g,\Omega_k(-\overline{z},\,\cdot\,)\rangle$ and the Shintani lift of $f$ is essentially $\langle f, \Omega_k(\,\cdot\, ,-\overline{\tau})\rangle$ (see \cite{kohnen85}).

In the following we consider a related generating function, where several complications arise.
As usual, we denote by $j$ the unique weakly holomorphic modular function for $\Gamma$ with a simple pole at the cusp $i\infty$.
Its derivative $j'$ is a weakly holomorphic modular form of weight two for $\Gamma$. It is well-known that $\Gamma$ acts on $\calQ_{-d}$ with finitely many orbits if $d \neq 0$.
For each positive $d$, we consider the meromorphic modular form of weight two given by
 \begin{equation*}
    F_d(z) := -2i\sum_{Q \in \Gamma\backslash\calQ_{-d}}\frac{1}{\omega_{Q}}\frac{j'(z)}{j(z)-j(z_Q)},
  \end{equation*}
where $\omega_{Q}$ is the size of the stabilizer of $Q$ in $\PSL_2(\Z)$ and $z_Q$ is the complex multiplication
(CM) point associated with $Q$, i.e., $z_Q$ is the unique root of $Q(z,1)$ contained in $\H$.
These functions are CM traces of
\begin{equation}\label{denominators}
\frac{j'(z)}{j(z)-j(\tau)} = -2\pi i \sum_{n=0}^\infty j_n(\tau) e(nz),
\end{equation}
where $j_n$ for $n \in \N_0$ is the unique weakly holomorphic modular function for $\Gamma$ with Fourier expansion $j_n(\tau)= e(-n\tau) + O(e(\tau))$, and the equality holds for $\Im(z) > \Im(\tau)$. Note that \eqref{denominators} is equivalent to the famous denominator formula for the Monster Lie algebra.

In this paper, we are interested in the (formal) generating function
\[
  A(\tau, z) := \sum_{d > 0} F_d(z) e(d\tau).
\]
However, leaving convergence issues aside,
since $F_d$ has poles at all of the CM points of discriminant $-d$ and the set of all CM points is dense in $\uhp$,
the resulting function would be undefined
on a dense set in $\H$ and thus badly behaved.
In this article, we study how to ``complete'' such a formal generating function to converge everywhere on $\uhp \times \uhp$ to a smooth function which is modular in both variables. 
This is analogous to the modular completion of the mock theta functions in the work of Zwegers \cite{Zwegers:thesis},
which turn out to be harmonic Maass forms.
To state our result, we extend the definition of $F_d$ to include non-positive discriminants. 
To be more precise, for $d\in -\N_0$ such that $-d$ is a discriminant, we set
\begin{equation*}
F_d(z) :=
\begin{dcases}
\frac{2\pi}{6}E_{2}^{*}(z) & \text{ if } d=0 ,\\
0 & \text{ if } d<0,
\end{dcases}
\end{equation*}
where $E_{2}^{*}$ is the non-holomorphic Eisenstein series of weight $2$ for $\Gamma$ defined in \eqref{eq:E2} below.
Furthermore, we let $(z=x+iy, \tau =u +iv)$
\[
\widetilde{F}_d(z; v) := -2 \sum_{Q\in \calQ_{-d}\setminus\{0\}}\frac{Q_z}{Q(z,1)}\exp\left(-4\pi v\frac{|Q(z,1)|^2}{y^2}\right),
\]
where for $Q=[a,b,c]$, we set $Q_z:=\frac1y(a|z|^2+bx+c)$. We obtain the following result (see Theorem \ref{thm:LpreimR0Siegel} for a proof).

\begin{theorem} \label{introthm:Fdgenser}
For $d > 0$ the function $(z,v) \mapsto F_{d}(z) + \widetilde{F}_{d}(z,v)$ extends to a smooth function on $\H \times \R^{+}$ if we define its value at a CM point $z_{Q}$ of discriminant $-d$ by $\lim_{z \to z_{Q}}(F_{d}(z) + \widetilde{F}_{d}(z,v))$. Furthermore, the series
\[
  A^*(\tau,z) := \sum_{d \in \Z} \left(F_d(z)+\widetilde{F}_d(z; v)\right) e(d\tau),
\]
converges locally uniformly on $\H \times \H$ to a smooth function. It is modular of weight two for $\Gamma$ in $z$ and of weight $\frac{3}{2}$ in $\tau$ for $\Gamma_0(4)$.
\end{theorem}

We remark that the our completion $A^*(\tau, z)$ is closely related
to upcoming work of Bruinier, Funke and Imamoglu \cite{bif2},
in which the lift of meromorphic modular forms like \eqref{denominators}
against the Siegel theta function is studied.

Theorem~\ref{introthm:Fdgenser} immediately begs the question whether the Fourier coefficients of $A^{*}(\tau,z)$ in $z$ are also of interest. 
As it turns out, the very same function is also a completion of a formal generating function of another natural family of modular forms, which we describe in the following.
For a positive discriminant $D$ we denote by $g_{D}$ the unique weakly holomorphic modular form of weight $\frac32$ for $\Gamma_0(4)$ in Kohnen's plus space having principal part $e(-D \tau)$.
In their influential paper \cite{dit}, Duke, Imamoglu, and T\'oth considered the finite sum
\begin{equation}
  \label{eq:DITfin}
  g_0(\tau) - \sum_{0 < m \leq M} \sum_{n \mid m} n g_{n^2}(\tau) e(mz),
\end{equation}
where $g_0:=\mathcal H$ is Zagier's non-holomorphic Eisenstein series of weight $\frac32$ for $\Gamma_0(4)$ defined in \eqref{eq:H}. After taking the (regularized) inner product of this sum against $g_{D}$ for a positive non-square discriminant $D > 0$ the limit as $M \to \infty$ exists. They then show that this limit is the generating function of the $D$-th traces of cycle integrals of the modular functions $j_{m}$ and a modular integral of weight two with rational period function.
Without taking the inner product first, the limit as $M \to \infty$ in \eqref{eq:DITfin} does not exist.
However, we show that it can again be completed to a convergent generating function which is modular in both variables. It turns out that this generating function in fact equals $A^{*}(\tau,z)$, giving its Fourier expansion in $z$ (see Theorem \ref{thm:KMD} for a precise statement in the vector-valued setting).

The function $A^{*}$ satisfies differential equations under the Laplace operators, and it is related to the Siegel and the Kudla-Millson theta functions by the Maass lowering operator. In our particular situation, the two theta functions are explicitly given by
\begin{align*}
\Theta_{S}(\tau,z) &:= 4v\sum_{d \in \Z}\sum_{ Q \in \mathcal{Q}_{-d}}\exp\left(-4\pi v \frac{|Q(z,1)|^{2}}{y^{2}}\right)e(d\tau),\\
\Theta_{KM}(\tau,z) &:= \sum_{d \in \Z}\sum_{Q \in \mathcal{Q}_{-d}}\left(4vQ_z^2-\frac{1}{2\pi} \right)\exp\left(-4\pi v \frac{|Q(z,1)|^{2}}{y^{2}}\right)e(d\tau),
\end{align*}
and they are smooth functions in $\tau$ and $z$ which transform like modular forms of weight $-\frac{1}{2}$ and $\frac{3}{2}$ in $\tau$ for $\Gamma_{0}(4)$, respectively, and weight $0$ in $z$ for $\Gamma$.

\begin{proposition}\label{prop: differential equations KudlaMillson theta introduction}
The function $A^{*}$ satisfies the differential equations
	\begin{align*}
	4\Delta_{\frac{3}{2},\tau}\left( A^{*}(\tau,z) \right) = \Delta_{2,z}\left(A^{*}(\tau,z) \right),
	\end{align*}
	where $\Delta_{k}$ is the weight $k$ hyperbolic Laplace operator defined in \eqref{eq Laplace operator}, and
	\begin{align}
L_{\frac{3}{2},\tau}\left(A^{*}(\tau,z)\right) &= -\frac{1}{16\pi}R_{0,z}\left(\Theta_{S}(\tau,z)\right), \label{eq L32 ThetaKM introduction} \\
L_{2,z}\left(A^{*}(\tau,z)\right) &= \Theta_{KM}(\tau,z), \notag
\end{align}
where $L_{\frac{3}{2},\tau} := -2iv^2\frac{\partial}{\partial \overline{\tau}}$, $L_{2,z} := -2iy^2\frac{\partial}{\partial \overline{z}}$, and $R_{0,z} := 2i\frac{\partial}{\partial z}$ are Maass lowering and raising operators.
\end{proposition}
In Proposition \ref{prop: differential equations KudlaMillson theta},
we prove the corresponding identity in a vector-valued setting,
which immediately implies Proposition \ref{prop: differential equations KudlaMillson theta introduction}. For the proof we use a method from \cite{ehlensankaran} which yields distinguished $L_{k}$-preimages of a certain class of smooth automorphic forms of moderate growth, generalizing the surjectivity of the $\xi_{k}$-operator from harmonic Maass forms to holomorphic modular forms  \cite[Theorem~3.7]{bruinierfunke04}. To deal with the above generating function, the following simplified version is sufficient.
The reader is referred to Section \ref{sec:sect-maass-lower} for details.

\begin{proposition}[\cite{ehlensankaran}, Theorem~2.14]
	Let $F(z) = \sum_{m \in \Z}c_{F}(m,y)e(mz)$ be a smooth modular function for $\Gamma$ of moderate growth\footnote{Here and throughout, we mean by moderate growth that $F$ and all iterated partial derivatives in $z$ and $\overline{z}$ are $O(y^\ell)$ for some $\ell \in \Z$ ($\ell$ and the implied constant are allowed to depend on the order of the partial derivative) as $y \to \infty$.}, and assume that $c_{F}(0,y) = O(e^{-Cy})$ as $y \to \infty$ for some $C > 0$. Then there exists a unique smooth weight two modular form $F^{\#}(z) = \sum_{m \in \Z}c_{F^{\#}}(m,y)e(mz)$ with at most linear exponential growth at the cusp and $\lim_{y \to \infty}c_{F^{\#}}(m,y) = 0$ for $m < 0$, such that $L_{2}(F^{\#}) = F$. Its Fourier coefficients are given by
	\[
	c_{F^{\#}}(m,y) = \left\langle F_{m}-P_{m,y},\overline{F} \right\rangle,
	\]
	where $\langle \cdot,\cdot \rangle$ is the regularized Petersson inner product defined in \eqref{eq:peterssoninnerproduct}, $F_{m} = j_{m}$ for $m \in\N_0$ is the unique weakly holomorphic modular form for $\Gamma$ with Fourier expansion $e(-mz) + O(e(z)),$ $F_{m} = 0$ for $m \in-\N$, and $P_{m,y}$ is a truncated Poincar\'e series defined in \eqref{eq:cutoffpoincareseries}.
\end{proposition}

By Proposition 1.3, for fixed $\tau$, the function
\[
\Theta_{KM}^{\#}(\tau,z) := \sum_{m \in \Z}\left\langle F_{m}-P_{m,y},\overline{\Theta_{KM}(\tau,z)} \right\rangle e(mz)
\]
is a smooth weight two modular form in $z$ for $\Gamma$ which maps to $\Theta_{KM}(\tau,z)$ under $L_{2,z}$. For $m \in\N_0$ the inner product with $F_{m} = j_{m}$ is the Kudla-Millson theta lift studied in \cite{bruinierfunke06}. By \cite[Theorems~1.1 and 1.2]{bruinierfunke06}, we have
\[
\left\langle F_{m},\overline{\Theta_{KM}(\tau,z)} \right\rangle =
\begin{cases}
  2\mathcal{H}(\tau) & \text{if } m = 0, \\
 -2\sum_{n \mid m}ng_{n^{2}}(\tau) & \text{if } m > 0.
 \end{cases}
\]
The integral involving the truncated Poincar\'e series can be computed by unfolding against $P_{m,y}$. In this way we obtain the Fourier expansion in $z$. The modularity of $\Theta_{KM}(\tau,z)$ in $\tau$ implies that $\tau\mapsto\Theta_{KM}^{\#}(\tau,z)$ transforms like a modular form of weight $\frac{3}{2}$. The series converges locally uniformly and defines a smooth function on $\H \times \H$ (which is in fact quite difficult to prove). It can be differentiated termwise with respect to $\tau$, which then easily implies the relation \eqref{eq L32 ThetaKM introduction}. Equivalently stated, \eqref{eq L32 ThetaKM introduction} says that $\Theta_{KM}^{\#}(\tau,z)$ maps to a multiple of $R_{0,z}(\Theta_{S}(\tau,z))$ under $L_{\frac{3}{2}}$. Applying the same technique to construct a $L_{\frac{3}{2}}$-preimage of $R_{0,z}(\Theta_{S}(\tau,z))$ with respect to $\tau$, we obtain Theorem~\ref{introthm:Fdgenser}, and from the uniqueness we see that this preimage agrees up to a constant factor with $\Theta_{KM}^{\#}(\tau,z)$.

\subsection{Higher weight}
\label{sec:higher-weight}
While we focus on the generating functions in low weights in this paper,
we also consider lowering preimages of the Shintani and Millson theta functions
in higher weight. Our interest in these functions is twofold.
First of all, we obtain again completions of the formal generating functions
of very natural families of modular forms.
In particular, the construction yields completions of the generating functions
of the modular forms $f_{k+1,D,d}$ (defined above for $D=1$ and $d<0$),
which are holomorphic cusp forms for $d>0$ (and a holomorphic Eisenstein series for $d=0$)
but meromorphic modular forms for $d<0$.

Second, the modular completions we consider in higher weight feature an interesting phenomenon:
it really becomes clear that the preimages under the lowering operator that we obtain with the method of this paper should in fact be seen as a very close relative of theta functions build from Schwartz functions.
Namely, it is possible to write them as theta functions coming from a ``degenerate''
Schwartz function. We refer to Section \ref{sec:higherweight} for details.

To state one of our results in this direction,
let $k$ be a positive integer and let $D$ be a fundamental discriminant with $(-1)^{k}D < 0$.
For all $D$ and $d$, the function
\[
  f_{k+1,d,D}(z) := \sum_{Q \in \mathcal{Q}_{d|D|}\setminus \{0\}}\frac{\chi_{D}(Q)}{Q(z,1)^{k+1}}
\]
has weight $2k+2$, where $\chi_{D}$ denotes the usual genus character (see Section \ref{sec:theta}).
Define the series
\begin{align*}
B_k^*(\tau,z) &:= k!\frac{|D|^{\frac{k+1}{2}}}{\pi^{k+1}}\sum_{\substack{d \in \Z \\ (-1)^{k}d \equiv 0,3 \pmod{4}}}\left(\widetilde{f}_{k+1,-d,D}(v,z)-g_{k+1,d,D}(v,z)\right)e\left(d\tau\right)
\end{align*}
where
\begin{align*}
\widetilde{f}_{k+1,d,D}(v,z) &:= f_{k+1,d,D}(z)\begin{dcases}
1 & \text{if $d \leq 0$,} \\
\frac{\Gamma\left(k+\frac12,4\pi |d| v\right)}{\Gamma\left(k+\frac12\right)} & \text{if $d > 0$,}
\end{dcases}\\
g_{k+1,d,D}(v,z) &:= \frac{1}{k!}\sum_{Q \in \calQ_{-d|D|}\setminus\{0\}}\frac{\chi_{D}(Q)}{Q(z,1)^{k+1}}\Gamma\left(k+1,4\pi v \frac{|Q(z,1)|^{2}}{y^{2}|D|} \right).
\end{align*}
Using the same method as above, we obtain the following result.
\begin{theorem}\label{thm:heigherweightintroduction}
We have that $B^*(\tau,z)$ converges to a smooth function on $\H \times \H$ and is modular of weight $2k+2$ in $z$ for $\Gamma$ and modular of weight $-k + \frac{1}{2}$ in $\tau$ for $\Gamma_0(4)$.
\end{theorem}

For the proof we refer to Section \ref{sec:heigherweightshintani}.

\subsection{How this article is organized}
In Section \ref{sec:prelims}, we set up the notation for the rest of the paper and
recall important facts from the theory of harmonic Maass forms.
Section \ref{sec:sect-maass-lower} is then concerned with the $L_k$-preimages;
the results of this section are essentially contained in \cite{ehlensankaran}.
However, since we need special cases and also slightly stronger versions
for our very explicit results, we nevertheless provide a lot of details and proofs.
The four theta functions that we use are all introduced in Section \ref{sec:theta}.
In Sections \ref{sec:KM and Siegel preimage}, \ref{sec:L2zpreimages Millson and Shintani wt 0},
and \ref{sec:higherweight}, we construct the preimages of all theta functions and study their
analytic behaviour in detail.
The appendix contains growth estimates for families of weakly holomorphic modular forms,
which are needed in order to prove normal convergence of the completed generating function.
These results might be of independent interest but are also quite technical and
only included because, to our surprise, we could not find similar results in the literature.
We only provide details in weight $\tfrac32$, which is the most delicate regarding convergence,
but our arguments generalize to other weights.

\section*{Acknowledgements}
Using different techniques, special cases of the completions for higher weight were also obtained by the first author and Kane and Rolen in unpublished notes.

The research of the first author is supported by the Deutsche Forschungsgemeinschaft (DFG) Grant No. BR 4082/3-1. 
The third author was partially supported by DFG Grant No. BR-2163/4-1 and the LOEWE research unit USAG.

\section{Preliminaries on vector-valued harmonic Maass forms for the Weil representation}
\label{sec:prelims}
Consider the even lattice
\[
L := \left\{ \begin{pmatrix}-b & -c \\ a & b \end{pmatrix} : a,b,c \in \Z\right\}
\]
with the quadratic form $Q(\lambda) := \det(\lambda)$ and bilinear form $(\lambda,\nu) := -\tr(\lambda \nu)$. It has signature $(1,2)$ and level four, and its dual lattice equals
\[
L' = \left\{ \begin{pmatrix} -\frac{b}{2} & -c \\ a & \frac{b}{2} \end{pmatrix}: a,b,c \in \Z\right\}.
\]
Hence, $L'/L \cong \Z/2\Z$. The modular group $\Gamma$ acts on $L'$ and $L$ by conjugation $\gamma.\lambda := \gamma \lambda\gamma^{-1}$, and fixes the classes of $L'/L$. For a discriminant $D \in \Z$ let
\[
L_{-\frac{D}{4}} := \left\{\lambda \in L': Q(\lambda) = -\frac{D}{4}\right\}.
\]
Note that an element $\lambda\in L_{-\frac{D}4}$ corresponds to a binary quadratic form $[a,b,c]$ of discriminant $D$, and this identification is compatible with the actions of $\SL_{2}(\Z)$ on $L_{-\frac{D}4}$ and $\calQ_{D}$.

We let $\Mp_{2}(\R)$ be the metaplectic group, realized as the set of pairs $\widetilde{\gamma}=(\gamma,\phi)$ with $\gamma = \left(\begin{smallmatrix}a & b \\ c & d \end{smallmatrix}\right) \in \SL_{2}(\R)$ and $\phi: \H \to \C$ a holomorphic function with $\phi^{2}(\tau) = c\tau + d$. The group $\widetilde{\Gamma} := \Mp_{2}(\Z)$ is generated by the elements $\widetilde{T} := \left(\left(\begin{smallmatrix}1 & 1 \\ 0 & 1 \end{smallmatrix}\right),1\right)$ and $\widetilde{S} := \left( \left(\begin{smallmatrix}0 & -1 \\ 1 & 0 \end{smallmatrix}\right),\sqrt{\tau}\right)$. Let $\widetilde{\Gamma}_{\infty}$ be the subgroup of $\widetilde{\Gamma}$ generated by $\widetilde{T}$.

We let $\e_\mu$ with $\mu \in L'/L$ be the standard basis vectors of the group algebra $\C[L'/L]$. 
We frequently identify $L'/L$ with $\Z/2\Z$
and for $n \in \Z$ we use the notation $\e_n$
to denote $\e_{n \pmod {2}}$, where $\e_0$ corresponds to $\e_{0+L}$
and $\e_{1}$ corresponds to $\e_{\gamma + L}$ with $\gamma \in L' \setminus L$.
The group algebra is equipped with the natural inner product $\langle \e_{\mu},\e_{\nu}\rangle = \delta_{\mu,\nu}$ which is antilinear in the second variable. Furthermore, let $\varrho_{L}$ denote the associated \emph{Weil representation} of $\widetilde{\Gamma}$ which is defined by
\begin{align*}
 \varrho_{L}\bigs(\widetilde{T}\bigs) \e_{\mu} := e(Q(\mu)) \e_{\mu}, \qquad
 \varrho_{L}\bigs(\widetilde{S}\bigs) \e_{\mu} := \frac{\sqrt{i}}{\sqrt{|L'/L|}}\sum_{\nu \in L'/L} e(-(\mu,\nu)) \e_{\nu}.
\end{align*}
Moreover $\overline{\varrho}_{L}$ denotes the complex conjugate representation, which corresponds to the Weil representation attached to the lattice given by $L$ with the negative of the quadratic form.

For $k \in \frac{1}{2}+\Z$, define the \emph{weight $k$ slash operator} of $\Mp_{2}(\Z)$ on functions $f: \H \to \C[L'/L]$ by
\[
  f|_{k,\varrho_{L}}(\gamma,\phi)(\tau) := \phi^{-2k}(\tau)\varrho_{L}(\gamma,\phi)^{-1}f(\gamma \tau).
\]
The \emph{weight $k$ Laplace operator}
\begin{align}\label{eq Laplace operator}
\Delta_{k}:=-v^2\left(\frac{\partial^2}{\partial u^2}+\frac{\partial^2}{\partial v^2}\right)
+ikv\left(\frac{\partial}{\partial u}+i\frac{\partial}{\partial v}\right)
\end{align}
acts component-wise on smooth functions $f: \H \to \C[L'/L]$ and commutes with the weight $k$ slash-action.

We recall the definition of harmonic Maass forms from  \cite{bruinierfunke04}.

\begin{definition}
	A \emph{harmonic Maass form} of weight $k \in \frac{1}{2}+\Z$ for $\varrho_{L}$ is a twice continuously differentiable function $f: \H \to \C[L'/L]$ which satisfies the following conditions:
	\begin{enumerate}
		\item $\Delta_{k}(f) = 0$;
		\item $f|_{k,\varrho_{L}}(\gamma,\phi) = f$ for every $(\gamma,\phi) \in \widetilde{\Gamma}$;
		\item $f(\tau) = O(e^{Cv})$ as $v \to \infty$ for some constant $C > 0$, uniformly in $u$.
	\end{enumerate}
\end{definition}
	We denote the space of all harmonic Maass forms by $H_{k,\varrho_{L}}^!$. Furthermore, we let $H_{k,\varrho_{L}}$ be the subspace of harmonic Maass forms for which there exists a Fourier polynomial
	\[
	P_{f}(\tau) =\sum_{-\infty \ll n\leq 0}c_{f}^{+}(n)e(n\tau)
	\]
	with coefficients $c_f^+(n)\in\C[L'/L]$ such that $f(\tau)-P_{f}(\tau) = O(e^{-cv})$ as $v \to \infty$ for some $c > 0$, uniformly in $u$. The function $P_{f}$ is called the \emph{principal part} of $f$. The subspaces of weakly holomorphic modular forms (meromorphic modular forms which are holomorphic on $\H$), holomorphic modular forms and cusp forms are denoted by $M_{k,\varrho_{L}}^{!},M_{k,\varrho_{L}}$, and $S_{k,\varrho_{L}}$, respectively. Harmonic Maass forms of half-integral weight for $\overline{\varrho}_{L}$ and of integral weight for $\Gamma$ are defined analogously, and the corresponding spaces are denoted by $H_{k,\overline{\varrho}_{L}}^!$ (for $k \in \frac{1}{2}+\Z$) and $H_{k, \varrho_L}^!$ (for $k \in \Z$), respectively.

An element $f \in H_{k,\varrho_L}^!$ has a Fourier expansion of the shape
\begin{equation} \label{deff}
  f(\tau) = \sum_{n \in \Q} c_f(n, v) e(n\tau),
\end{equation}
with $c_f(n,v) \in \C[L'/L]$. The right hand side of \eqref{deff} decomposes into a \emph{holomorphic part} $f^+$ and a \emph{non-holomorphic part} $f^-$
as follows, which are for $k \neq 1$ given by
\begin{align*}
  f^+(\tau)&= \sum_{\substack{n\in \Q\\ n\gg-\infty}} c_{f}^+(n) e(n\tau),\qquad f^-(\tau) = c_{f}^-(0)v^{1-k} + \sum_{\substack{n\in \Q \setminus\{0\}
\\ n \ll \infty}} c_{f}^-(n) W_k(2\pi n v) e(n\tau) , \end{align*}
with coefficients $c_{f}^{+}(n),c_{f}^{-}(n) \in \C[L'/L]$. Here, following \cite{BDE} and  \cite{bruinierfunke04} for $x\in\mathbb{R}$, we set
\begin{equation*}
\label{defineW}
W_k(x):= (-2x)^{1-k} \Re(E_k(-2x))
\end{equation*}
with $E_k$ the \emph{generalized exponential integral}
(see  \cite{NIST:DLMF}, 8.19.3) defined by
\[
  E_r(z) := \int_1^{\infty} e^{-zt}t^{-r}\, dt.
\]
This function is related to the incomplete Gamma function via (8.19.1) of  \cite{NIST:DLMF} by $\Gamma(r, z) = z^r E_{1-r}(z)$. For $k = 1$, one has to replace $v^{1-k}$ by $\log(v)$ in the non-holomorphic part $f^{-}$ of $f$. Note that $f \in H_{k,\varrho_L}$ is equivalent to $c_f^-(n) = 0$ for all $n < 0$.

The \emph{Maass lowering operator} and the \emph{Maass raising operator} 

\begin{equation*}
  L_{k} := -2iv^{2}\frac{\partial}{\partial \overline{\tau}}, \qquad R_{k} := 2i\frac{\partial}{\partial \tau} +\frac{k}{v},
\end{equation*}
lower or raise the weight of an automorphic form of weight $k$ by $2$. The \emph{$\xi_k$-operator}
\[
\xi_{k}(f(\tau)) := v^{k-2}\overline{L_{k}(f(\tau))} = R_{-k}\left(v^{k}\overline{F(\tau)}\right) = 2iv^{k}\overline{\frac{\partial}{\partial \overline{\tau}}f(\tau)}
\]
defines surjective maps $H_{k,\varrho_{L}}^! \to M_{2-k,\overline{\varrho}_{L}}^{!}$ and $H_{k,\varrho_{L}} \to S_{2-k,\overline{\varrho}_{L}}$. The raising and lowering operators are related to the Laplace operator by
\begin{align}\label{eq laplace via lowering raising}
-\Delta_{k} = \xi_{2-k}\circ \xi_{k} = L_{k+2}\circ R_{k}+k = R_{k-2} \circ L_{k}.
\end{align}

\begin{remark}\label{rem:vectorscalarvaluedisomorphism}
	The action $\varrho_{L}(Z)\e_{h} = i\e_{-h}$ of $Z := \widetilde{S}^{2}$ in the Weil representation implies that the components $f_{h}$ of $f = \sum_{h \in L'/L}f_{h}\e_{h} \in H_{k,\varrho_{L}}^!$ satisfy the symmetry relation $f_{-h} = (-1)^{k+\frac12}f_{h}$. We obtain that $H_{k,\varrho_{L}}^! = \{0\}$ if $k+\frac12$ is odd and that $H_{k,\overline{\varrho}_{L}}^! = \{0\}$ if $k+\frac12$ is even. Denote by $H_{k}^!(4)$ the space of scalar-valued harmonic Maass forms $f(\tau) = \sum_{n \in \Z}c_{f}(n,v)e(n\tau)$ of weight $k$ for $\Gamma_{0}(4)$ satisfying the \emph{Kohnen plus space condition} $c_f(n,v)=0$ unless $ (-1)^{k-\frac12}n \equiv 0,1 \pmod 4$. Then by  \cite[Theorem 5.4]{eichlerzagier} the map
	\[
	f_{0}(\tau)\e_{0} + f_{1}(\tau)\e_{1} \mapsto f_{0}(4\tau) + f_{1}(4\tau)
	\]
	defines an isomorphism $H_{k,\varrho_{L}}^! \cong H_{k}^!(4)$ if $k+\frac{1}{2}$ is even, and $H_{k,\overline{\varrho}_{L}}^! \cong H_{k}^!(4)$ if $k+\frac{1}{2}$ is odd. Throughout this work, we switch freely between the vector-valued and the scalar-valued viewpoint without further notice. In particular, we use the same symbol for a vector-valued harmonic Maass form and its scalar-valued version.
\end{remark}

\begin{example}
We collect some examples of harmonic Maass forms and modular forms that are used below.
  \begin{enumerate}
              \item Zagier's non-holomorphic Eisenstein series
	\begin{align}\label{eq:H}
   \mathcal H(\tau) := \sum_{n=0}^{\infty}H(n)e(n\tau) + \frac{1}{4\sqrt{\pi}}\sum_{n=1}^\infty n \Gamma\left(-\frac12, 4\pi n^2v\right)e(-n^2\tau)+\frac{1}{8\pi \sqrt{v}},
	\end{align}
	with
	\[
	H(0) := -\frac{1}{12}, \quad  H(d) := \sum_{Q \in \Gamma \setminus \calQ_{-d}}\frac{1}{\omega_{Q}}
	\]
	is a harmonic Maass form in $H_{\frac{3}{2}}^!(4) \cong H_{\frac{3}{2},\varrho_{L}}^!$ (see  \cite{zagiereisenstein}). It is related to the Jacobi theta function $\theta(\tau) := \sum_{n \in \Z}e(n^{2}\tau) \in M_{\frac{1}{2}}(4) \cong M_{\frac{1}{2},\overline{\varrho}_{L}}$ by $\xi_{\frac{3}{2}}(\mathcal{H}) = -\frac{1}{16\pi}\theta$.
      \item For each negative discriminant $-d < 0$, there exists a unique weakly holomorphic modular form $f_{d} \in M_{\frac12}^{!}(4)$ having a Fourier expansion of the shape
	\begin{align*}
	f_{d}(\tau) = e(-d\tau) + \sum_{\substack{D > 0 \\ D \equiv 0,1 \pmod{4}}}A(D,d)e(D\tau).
	\end{align*}
	Similarly, for each positive discriminant $D > 0$ there is a unique weakly holomorphic modular form $g_{D} \in M_{\frac32}^{!}(4)$ with
	\begin{align}\label{eq:gD}
	g_{D}(\tau) = e(-D\tau) + \sum_{\substack{d \geq 0 \\ d \equiv 0,3 \pmod{4}}}B(D,d)e(d\tau).
	\end{align}
	Here $B(D,0) = -2$ if $D$ is a square, and $B(D,0) = 0$ otherwise. If we define $f_{0} := \theta$, then the sets $\{f_{d}\}$ and $\{g_{D}\}$ form bases of $M_{\frac12}^{!}(4)$ and $M_{\frac32}^{!}(4)$, respectively. The coefficients satisfy the \emph{Zagier duality} $A(D,d) = -B(D,d)$ and can be expressed in terms of twisted traces of CM values of the modular $j$-function (see  \cite{zagiertraces}).
      \item The non-holomorphic Eisenstein series ($\sigma_j(n):=\sum_{d|n}d^j$ is the $j$-th divisor sum)
	\begin{align}\label{eq:E2}
	E_{2}^{*}(z) := -\frac{3}{\pi y} + 1 -24\sum_{n=1}^{\infty}\sigma_{1}(n)e(nz)
	\end{align}
	is a harmonic Maass form of weight two for $\Gamma$. It satisfies $\xi_{2}(E_{2}^{*}) = \frac{3}{\pi}$.
  \end{enumerate}
\end{example}

\section{Normalized $L_k$-preimages}
\label{sec:sect-maass-lower}
In this section, we recall a special case of a result from \cite{ehlensankaran}, which produces a
distinguished preimage $F \in L_{k+2}^{-1}(f)$ of an automorphic form $f$ of weight $k$ under the $L_{k+2}$-operator.
We formulate this result for scalar-valued integral
weight modular forms on the full modular group $\Gamma = \SL_2(\Z)$
and for half-integral weight vector-valued modular forms for the Weil representation $\varrho_L$ (or $\overline{\varrho}_L$).
To ease the notation, we deal with the two cases separately.

\subsection{Integral weight}
\label{sec:integral-weight}
For $k \in \Z$, we define a family of harmonic Maass forms $F_m \in H_{k}$ ($m \in \Z$). For $m > 0$, we let $F_m \in H_{k}$ be a harmonic Maass form with principal part $e(-mz) + c_{F_{m}}^{+}(0)$ for some constant $c_{F_{m}}^{+}(0) \in \C$, which is unique up to addition of holomorphic modular forms. If $M_{k} \neq \{0\}$, then we let $F_{0} = 1$ if $k = 0$ and $F_{0} = E_{k}$ (the normalized Eisenstein series) if $k \neq 0$, and we additionally require that $c_{F_{m}}^{+}(0) = 0$ for all $m > 0$. If $M_{k} = \{0\}$, then we set $F_{m} = 0$ for $m \leq 0$.

\begin{example}
	For $k = 0$ and $m \geq 0$ the function $F_{m}$ is the unique weakly holomorphic modular function whose Fourier expansion has the form $e(-mz)+O(e(z))$. It is usually denoted by $j_{m}$ and it is a polynomial in the modular $j$-function, e.g., $j_{0} = 1, j_{1} = j-744$.
\end{example}

Furthermore, for $m \in \Z$ and $w \in \R^+$, we define the truncated Poincar\'e series
\begin{align}\label{eq:cutoffpoincareseries}
  P_{m,w}(z) := \frac{1}{2}\sum_{\gamma \in \Gamma_{\infty}\backslash \Gamma}\big(\sigma_{w}(y)e(-mz)\big)\big|_{k} \gamma, \quad \text{ where }\sigma_{w}(y) := \begin{cases}1 &\quad \textnormal{if } y \geq w, \\ 0 &\quad \textnormal{if } y < w, \end{cases}
\end{align}
with $|_{k}$ the usual weight $k$ slash operator, and $\Gamma_{\infty} := \left\{\left(\begin{smallmatrix}1 & n \\ 0 & 1 \end{smallmatrix} \right): n \in \Z\right\}$.

We introduce some more notation.
\begin{definition}\label{dfn:Amod}
The space $\Amod{k}$ consists of all smooth functions $F: \H \to \C$ satisfying the following conditions:
	\begin{enumerate}
		\item $F\vert_{k} \gamma = F$ for all $\gamma \in \Gamma$;
		\item $\frac{\partial^{\alpha}}{\partial z^{\alpha}}\frac{\partial^{\beta}}{\partial \overline{z}^{\beta}}F(z) = O(y^{\ell_{\alpha + \beta}})$ for some $\ell_{\alpha+\beta} \in \N_0$ as $y \to \infty$, uniformly in $x$, for all $\alpha,\beta \in \N_{0}$.
		\item If $F(z) = \sum_{m \in \Z}c_{F}(m,y)e(mz)$ denotes the Fourier expansion of $F$, then the integral
\begin{align*}
  \label{eq:0term}
  \int_{1}^{\infty}c_{F}(0,t)t^{-2-s}dy,
\end{align*}
which converges for $\Re(s) \gg 0$ large enough, has a meromorphic continuation to a half-plane $\Re(s) > -\varepsilon$ for some
$\varepsilon > 0$.
	\end{enumerate}
\end{definition}

\begin{remark}
	The integral in Definition \ref{dfn:Amod} (iii) converges for $\Re(s) \gg 0$ since $F$ is of polynomial growth as $y \to \infty$. The condition that it has a meromorphic continuation to $s = 0$ ensures that all of the regularized Petersson inner products (see \eqref{eq:peterssoninnerproduct} below) appearing in the following are well-defined. Also note that we modify the definition of the space $\Amod{k}$ slightly in comparison to \cite{ehlensankaran}.
\end{remark}

Next we define a space of smooth modular forms in which the $L_{k+2}$-preimages of functions in $\Amod{k}$ live.

\begin{definition}
The space $\Aexp{k+2}$ consists of all smooth functions $G: \H \to \C$ such that
	\begin{enumerate}
		\item $G\vert_{k+2} \gamma = G$ for all $\gamma \in \Gamma$;
		\item $G(z) = O(e^{Cy})$ as $y \to \infty$ for some constant $C>0$, uniformly in $x$;
		\item $L_{k+2}(G) \in \Amod{k}$.
	\end{enumerate}
	\end{definition}

We next describe some basic properties of the Fourier coefficients of functions in $\Aexp{k+2}$.

\begin{lemma}\label{lemma growth of coefficients}
Let $G(z) = \sum_{m \in \Z}c_{G}(m,y)e(mz) \in \Aexp{k+2}$. For $m \in-\N$ the limit
\[
  \kappa_G(m):=\lim_{y \to \infty} c_G(m,y)
\]
exists and vanishes for all but finitely many $m \in-\N$. For $m \in\N$, we have the estimate
\[
c_G(m,y) = O\left(y^{\ell}e^{2\pi m y}\right)
\]
as $y \to \infty$, for some $\ell \in \N_0$.
\end{lemma}

\begin{proof}
 Let $m \in \Z$. If $L_{k+2}(G) = F \in \Amod{k}$, then we can write
 \[
   c_G(m,y) = c_G(m,1) + \int_1^yc_F(m,t)t^{-2}dt.
 \]
 Since $F(z)=O(y^\ell)$ for some $\ell \in \N_0$, we have $\abs{c_F(m,y)} \leq Cy^\ell e^{2\pi m y}$ for some $C>0$. This easily implies the statements of the lemma.
\end{proof}

Following \cite{borcherds}, we define the \textit{regularized Petersson inner product} of two automorphic forms $F$ and $G$ of weight $k$ by
\begin{equation} \label{eq:peterssoninnerproduct}
  \langle F,G \rangle := \CT_{s = 0}\lim_{T \to \infty}\int_{\mathcal{F}_{T}}F(z)\overline{G(z)} y^{k-s}\frac{dx dy}{y^{2}},
\end{equation}
where $\mathcal{F}_{T} := \{z \in \H: |x| \leq \frac12, |z| \geq 1, y \leq T\}$ is a truncated fundamental domain for $\Gamma \backslash \H$ and $\CT_{s = 0}f(s)$ denotes the constant term in the Laurent expansion at $s = 0$ of the meromorphic continuation of a function $f$.
Of course, the inner product \eqref{eq:peterssoninnerproduct} does not always exist, but this discription is sufficient for the purposes of this work. In particular, $\langle F, G \rangle$
always exists for $F \in \Aexp{k}$ and $G \in M_k$ (see \cite[Lemma~2.10]{ehlensankaran} for a proof).
For notational convenience we define a bilinear pairing between automorphic forms $F$ and $G$ of weight $k$ and $-k$, respectively, by
\[
  \bil{F}{G} := \CT_{s = 0}\lim_{T \to \infty}\int_{\mathcal{F}_{T}}F(z)G(z)y^{-s}\frac{dx  dy}{y^{2}}
              = \left\langle F,y^{-k}\overline{G}\right\rangle,
\]
whenever this exists.

The following result (essentially a special case of \cite[Proposition 2.12]{ehlensankaran}) yields distinguished $L_{k+2}$-preimages of certain functions in $\Amod{k}$.
\begin{proposition}\label{prop:Lsharpint}
Let $F \in \Amod{k}$, such that $\bil{F}{G} = 0$ for all $G \in S_{-k}$.
Then there exists a unique $F^{\#} \in \Aexp{k+2}$ with Fourier expansion
\[
  F^{\#}(z) = \sum_{m \in \Z}c_{F^{\#}}(m,y)e(mz),
\]
such that
\begin{enumerate}
\item $L_{k+2}(F^{\#}) = F$;
\item $\kappa_{F^{\#}}(m) = 0$ for all $m<0$;
\item if $M_{k+2} \neq \{0\}$, then
\begin{equation}\label{eq:ctn}
c_{F^{\#}}(0,y) =  - \CT_{s = 0}\int_{y}^{\infty}c_{F}(0,t)t^{-s-2}dt;
\end{equation}
\item  $F^{\#}$ is orthogonal to cusp forms, i.e., $\langle F, G \rangle = 0$ for all $G \in S_{k+2}$.
\end{enumerate}
The Fourier coefficients of $F^{\#}$ are given by
\[
  c_{F^{\#}}(m, y) = \bil{F_m - P_{m,y}}{F}.
\]
\end{proposition}

\begin{remark}
	Note that
        \begin{align*}
          L_{k+2}\left(-\CT_{s = 0}\int_{y}^{\infty}c_{F}(0,t)t^{-s-2}dt\right)
         &= y^{2}\frac{\partial}{\partial y}\left(-\CT_{s = 0}\int_{1}^{\infty}c_{F}(0,t)t^{-s-2}dt + \int_{1}^{y}c_{F}(0,t)t^{-2}dt\right)  = c_{F}(0,y).
        \end{align*}
	Thus, the constant term of any $L_{k+2}$-preimage of $F$ and the expression on the right-hand side of \eqref{eq:ctn} can only differ by a constant. In the case $M_{k+2} \neq \{0\}$, we normalize this constant to be zero.
\end{remark}

\begin{proof}[Proof of Proposition \ref{prop:Lsharpint}]
	For the convenience of the reader we only use the existence of some function $F^{\#} \in \Aexp{k+2}$ with $L_{k+2}(F^{\#}) = F$ from \cite[Proposition 2.12]{ehlensankaran}, and sketch the proofs for the remaining claims, which can be found in a more detail in \cite[Theorem 2.14]{ehlensankaran}. For simplicity, we assume that $M_{-k} = \{0\}$ and thus $M_{k+2} \neq \{0\}$; the case that $M_{-k} \neq \{0\}$ and that $M_{k+2} = \{0\}$ is completely analogous.

	We first show how the normalization in (ii), (iii), and (iv) can be achieved. An application of Stokes' Theorem as in the proof of \cite[Proposition~3.5]{bruinierfunke04} shows that for every cusp form $G(z) = \sum_{m > 0}c_{G}(m)e(mz) \in S_{-k}$ we have the formula
	\begin{align*}
	0 = \bil{F}{G} = \bil{L_{k+2}\bigs(F^{\#}\bigs)}{G} = \sum_{m > 0}c_{G}(m)\kappa_{F^{\#}}(-m).
	\end{align*}
	By \cite[Theorem~3.1]{borcherdsgkz} this implies that there exists a weakly holomorphic modular form of weight $2-k$ with principal part $\sum_{m < 0}\kappa_{F^{\#}}(m)e(mz)$. Subtracting it from $F^{\#}$ we can assume that $\kappa_{F^{\#}}(m) = 0$ for all $m < 0$. Furthermore, by subtracting a suitable cusp form from $F^{\#}$ we can assume that $F^{\#}$ is orthogonal to cusp forms, and by subtracting a suitable constant (if $k = -2$) or a suitable multiple of an Eisenstein series we obtain the normalization in (iii). It is clear that the properties (i)--(iv) determine $F^{\#}$ uniquely.

	We next compute the Fourier coefficients of $F^{\#}$.
        For $m \in \Z$ a standard argument (see e.g. \cite[Proposition 2.14]{ehlensankaran}) shows that
	\begin{align}
	\begin{split}\label{eq:alternativeregularization Fm}
	\bil{F_{m}-P_{m,y}}{F} &= \lim_{T \to \infty}\bigg(\int_{\mathcal{F}_{T}}(F_{m}(z_{1})-P_{m,y}(z_{1}))F(z_{1})\frac{dx_{1} dy_{1}}{y_{1}^{2}} \\
	&\qquad \qquad \qquad + \left(c_{F_{m}}^{+}(0)-\delta_{m = 0}\right)\CT_{s = 0}\int_{T}^{\infty}c_{F}(0,y_{1})y_{1}^{-s-2}dy_{1}\bigg).
	\end{split}
	\end{align}
	For $m > 0$, Stokes' theorem yields
	\begin{align*}
	\int_{\mathcal{F}_{T}}F_{m}(z_{1})F(z_{1})\frac{dx_{1}dy_{1}}{y_{1}^{2}} &= -\int_{\mathcal{F}_{T}}F^{\#}(z_{1})\overline{\xi_{-k}(F_{m}(z_{1}))}y_{1}^{k+2}\frac{dx_{1} dy_{1}}{y_{1}^{2}} + \int_{iT}^{1+iT}F_{m}(z_{1})F^{\#}(z_{1})dz_{1}.
	\end{align*}
	The first integral on the right-hand side vanishes as $T \to \infty$ since $F^{\#}$ is orthogonal to cusp forms. The remaining path integral can be evaluated as
	\begin{multline*}
	 c_{F^{\#}}(m,T)+ c^{+}_{F_{m}}(0)c_{F^{\#}}(0,T)+ \sum_{n = 1}^{n_{0}}c^{+}_{F_{m}}(n)c_{F^{\#}}(-n,T) +  \sum_{n = -n_{0}}^{-1}c^{-}_{F_{m}}(n)c_{F^{\#}}(-n,T)W_{k}(2\pi nT) \\
	+\int_{iT}^{1+iT}\sum_{n \geq n_{0}+1}c_{F_{m}}^{+}(n)e(nz_{1})F^{\#}(z_{1})dz_{1} + \int_{iT}^{1+iT}\sum_{n \leq -n_{0}-1}c_{F_{m}}^{-}(n)W_{k}(2\pi|n|y_{1})e(nz_{1})F^{\#}(z_{1})dz_{1},
	\end{multline*}
	where $n_{0}$ is any positive integer chosen such that the integrands in the last two integrals are exponentially decreasing as $y_{1} \to \infty$, which is possible since $F^{\#}$ is of linear exponential growth. In particular, the last two integrals vanish as $T \to \infty$. The finite sums in the second line vanish as $T \to \infty$ since $\lim_{T \to \infty}c_{F^{\#}}(-n,T) = \kappa_{F^{\#}}(-n) = 0$ for $n > 0$ and $c_{F^{\#}}(-n,T) = O(T^{\ell}e^{-2\pi n T})$ for $n < 0$  (see Lemma~\ref{lemma growth of coefficients}), and  $W_{k}(2\pi nT)= O(T^{-k}e^{4\pi n T})$ for $n < 0$. Using the normalization of $c_{F^{\#}}(0,T)$ in item (3) we see that the term $c_{F_{m}}^{+}(0)c_{F^{\#}}(0,T)$ cancels out with one of the extra terms coming from the regularization in \eqref{eq:alternativeregularization Fm}. In particular, in the limit $T \to \infty$ only the term $c_{F^{\#}}(m,T)$ gives a contribution.

	Furthermore, for $m \in \Z$ and for sufficiently large $T$ the unfolding argument gives
	\begin{align*}
	\int_{\mathcal{F}_{T}}P_{m,y}(z_{1})F(z_{1})\frac{dx_{1} dy_{1}}{y_{1}^{2}} &= \int_{y}^{T}\int_{0}^{1}e(-mz_{1})F(z_{1})\frac{dx_{1} dy_{1}}{y_{1}^{2}} \\
	 &= \int_{y}^{T}c_{F}(-m,y_{1})e^{4\pi m y_{1}}y_{1}^{-2}dy_{1} = c_{F^{\#}}(m,T) - c_{F^{\#}}(m,y).
	\end{align*}
	Combining everything, we obtain the stated formula for $c_{F^{\#}}(m,y)$.	\end{proof}

\subsection{Half-integral weight}
\label{sec:half-integral-weight}

We next explain how the analogs of the families $F_{m}$ and $P_{m,y}$ of harmonic Maass forms and the truncated Poincar\'{e} series are defined in the vector-valued case. For simplicity, we only do this for the representation $\varrho_{L}$, since the case $\overline{\varrho}_{L}$ can be treated in the same way.

Let $k \in \frac{1}{2}+\Z$. The spaces $\Amod{k,\varrho_{L}}$ and $\Aexp{k,\varrho_{L}}$ are defined in the obvious way. We can assume that $k+\frac{1}{2}$ is even since otherwise $H_{k,\varrho_{L}} = \{0\}$. For $D > 0$ with $D \equiv 0,1 \pmod{4}$ we let $\mathcal{F}_{D} \in H_{k,\varrho_{L}}^{+}$ be any harmonic Maass form with principal part $e(-\frac{D\tau}{4})\e_{D} + c_{\mathcal{F}_{D}}^{+}(0)\e_{0}$, which is unique up to addition of holomorphic modular forms. If there exists a holomorphic modular form in $M_{k, \varrho_{L}}$ which has constant term equal to one, then we take this as $\mathcal{F}_{0}$, and we further require that $c_{\mathcal{F}_{D}}^+(0) = 0$ for all $D > 0$. If there is no such holomorphic modular form, then we set $\mathcal{F}_0 = 0$. Finally, for $D < 0$, we let $\mathcal{F}_{D} = 0$.

\begin{example}
	If $k = \frac{3}{2}$, then we take the weakly holomorphic forms $\mathcal{F}_{D} = g_{D}$ for $D > 0$ and $\mathcal{F}_{D} = 0$ for $D \leq 0$.
\end{example}

We also define the truncated Poincar\'e series
\begin{align*}
\mathcal{P}_{D,w}(\tau) := \frac{1}{4}\sum_{\gamma \in \widetilde{\Gamma}_{\infty}\setminus \widetilde{\Gamma}}\left(\sigma_{w}(v)e\left(-\frac{D\tau}{4}\right) \e_{D}\right)\bigg|_{k, \varrho_{L}}\, \gamma.
\end{align*}

With these definitions, the following analog of Proposition~\ref{prop:Lsharpint} holds in the vector-valued case.

\begin{proposition}\label{prop:Lsharphalfint}
Let $f \in \Amod{k,\varrho_{L}}$, such that $\bil{f}{g} = 0$ for all $g \in S_{-k,\overline{\varrho}_{L}}$.
Then there exists a unique $f^{\#} \in \Aexp{k+2,\varrho_{L}}$ with Fourier expansion
\[
  f^{\#}(\tau) = \sum_{\substack{D \in \Z \\ D \equiv 0,1 \pmod 4}}c_{f^{\#}}(D,v)e(D\tau),
\]
such that the following conditions hold:
\begin{enumerate}
\item $L_{k+2}(f^{\#}) = f$;
\item $\kappa_{f^{\#}}(D) = 0$ for all $D<0$.
\item If $M_{k+2,\varrho} \neq \{0\}$, then
\begin{equation}\label{eq:ctn2}
c_{f^{\#}}(0,v) =  - \CT_{s = 0}\int_{v}^{\infty}c_{f}(0,t)t^{-s-2}dt;
\end{equation}
\item  $f^{\#}$ is orthogonal to cusp forms, i.e., $\langle f, g \rangle = 0$ for all $g \in S_{k+2,\varrho_{L}}$.
\end{enumerate}
The Fourier coefficients of $f^{\#}$ are given by
\[
  c_{f^{\#}}(D, v) = \bil{\mathcal{F}_D - \mathcal{P}_{D,v}}{f}\e_{d}.
\]
\end{proposition}

\section{Theta functions}
\label{sec:theta}
In this section we briefly recall the main players of this article, well-known non-holomorphic theta functions for the lattice $L$.

For $z = x+iy \in \H$ and $\lambda \in L'$, we define the quantities
\begin{align*}
Q_{\lambda}(z) := az^{2} +bz + c, \quad p_{\lambda}(z) := -\frac{1}{y}(a|z|^{2} + bx + c),\quad R(\lambda,z) := \frac{1}{2}p_{\lambda}^{2}(z) - (\lambda,\lambda).
\end{align*}
Note that $(\lambda,\lambda)=-\frac12(b^2-4ac)$ for $\lambda \in L'$. Let $D$ be a fundamental discriminant (possibly $1$). For a quadratic form $Q = [a,b,c] \in \calQ_{\Delta}$ of discriminant $\Delta$ such that $D$ divides $\Delta$, we let
\[
\chi_{D}(Q) := \begin{dcases}
\left(\frac{D}{n}\right) & \text{if $\gcd(a,b,c,D) = 1$ and $Q$ represents $n$ with $\gcd(n,D) = 1$}, \\
0 &\text{otherwise},
 \end{dcases}
\]
be the \emph{genus character} as in \cite[Section 1]{kohnen85}. We can view $\chi_{D}$ as a function on $L'/L$ via the identification of elements $\lambda \in L'$ with binary quadratic forms. Define
\[
\widetilde{\varrho}_{L} := \begin{cases} \varrho_{L}  & \text{if } D > 0, \\ \overline{\varrho}_{L} & \text{if } D < 0,  \end{cases}  \qquad \varepsilon(D) := \begin{cases}1 & \text{if } D > 0, \\ i & \text{if } D < 0. \end{cases}
\]

For a rapidly decaying (Schwartz) function $\varphi: L \otimes_\Z \R \to \C$ we consider the following twisted vector-valued theta function
\[
  \Theta(\varphi) := \sum_{\mu \in L'/L} \sum_{\substack{\lambda \in L + D\mu \\ Q(\lambda)\, \equiv\, D Q(\mu) \pmod{D}}} \chi_{D}(\lambda) \varphi(\lambda) \e_\mu.
\]
The Schwartz functions which we use are defined as follows.
We always write $\varphi(\lambda; \tau, z)$,
also depending on $\tau$ and $z$, as a product $\varphi(\lambda; \tau, z) = \varphi^0(\lambda; \tau)\varphi_D^\infty(\lambda; \tau, z)$,
with the Gaussian
\[
  \varphi^\infty_D(\lambda; \tau, z) := \exp\left(\frac{-2\pi v (Q(\lambda) + R(\lambda,z))}{|D|}\right).
\]
The four cases which we need for $\varphi^0$ are
\begin{align*}
\varphi_{S,D}^0(\lambda; \tau) &:= v e\left(\frac{Q(\lambda)u}{|D|}\right),&\quad
\varphi^0_{KM,D}(\lambda; \tau) &:= \left(\frac{vp_{\lambda}^{2}(z)}{|D|} - \frac{1}{2\pi} \right) e\left(\frac{Q(\lambda)u}{|D|}\right),\\
\varphi^0_{Sh,k,D}(\lambda; \tau) &:= v^{k+2} \frac{Q_{\lambda}^{k+1}(\overline{z})}{y^{2k+2}|D|^{\frac{k+1}{2}}} e\left(\frac{Q(\lambda)u}{|D|}\right),&\quad\varphi^0_{M,k,D}(\lambda; \tau) &:= v^{k+1} \frac{p_{\lambda}(z)Q_{\lambda}^{k}(\overline{z})}{y^{2k}|D|^{\frac{k+1}{2}}} e\left(\frac{Q(\lambda)u}{|D|}\right).
\end{align*}
The corresponding theta functions are
then defined as $\Theta_{S,D}^0 := \Theta(\varphi_{S,D}^0\varphi_D^\infty)$, $\Theta_{KM,D} := \Theta(\varphi_{KM,D}^0\varphi_D^\infty)$,
$\Theta_{Sh,k,D} := \Theta(\varphi_{Sh,k,D}^0\varphi_D^\infty)$,
and $\Theta_{M,k,D} := \Theta(\varphi_{M,k,D}^0\varphi_D^\infty)$ and
are called the \emph{Siegel}, \emph{Kudla-Millson}, \emph{Shintani}, and \emph{Millson theta functions}, respectively.
For simplicity, we drop $k$ from the notation when $k=0$.

\begin{remark} By replacing $\lambda$ with $-\lambda$ in the summations, we see that the Siegel and Kudla-Millson theta function vanish for $D < 0$, whereas the Shintani and Millson theta functions vanish for $(-1)^{k}D > 0$.
\end{remark}

We summarize their transformation properties in the following proposition.

\begin{proposition}\label{prop: transformation rules theta functions}
	The theta functions defined above have the following automorphic properties:
        \begin{enumerate}
        \item $\Theta_{S,D}(\cdot,z) \in \Amod{-\frac{1}{2}, \widetilde{\varrho}_L}$
                      and the components of $\Theta_{S,D}(\tau,\cdot)$ are contained in $\Amod{0}$;
       \item $\Theta_{KM,D}(\cdot,z) \in \Amod{\frac{3}{2}, \widetilde{\varrho}_L}$ and the components of $\Theta_{KM,D}(\tau,\cdot)$ are contained in $\Amod{0}$;
       \item $\Theta_{Sh,k,D}(\cdot,z) \in \Amod{-\frac{3}{2}-k, \widetilde{\varrho}_L}$ and the components of $\Theta_{Sh,k,D}(\tau,\cdot)$ are contained in $\Amod{2k+2}$;
       \item $\Theta_{M,k,D}(\cdot,z) \in \Amod{\frac{1}{2}-k, \widetilde{\varrho}_L}$ and the components of $\Theta_{M,k,D}(\tau,\cdot)$ are contained in $\Amod{2k}$.
        \end{enumerate}
\end{proposition}

\begin{proof}
	For the twisted Siegel theta function see \cite[Theorem 4.1]{bruinierono}. The
  transformation properties of the Millson and Shintani theta functions are
  stated in \cite[Theorems 3.1 and 3.3]{agor}. We remark that the Millson theta
  function has been studied earlier in \cite[Satz 2.8]{hoevel}. The twisted
  Kudla-Millson theta function is investigated in
  \cite[Proposition~4.1]{alfesehlen}. In general, the transfomation rules of
  such theta functions can be proved for $D = 1$ by Poisson summation as in
  \cite[Theorem 4.1]{borcherds}, and for $D \neq 1$ by using the twisting
  operator introduced in \cite{alfesehlen}.
\end{proof}

The theta functions are related by various differential equations.

\begin{proposition}\label{prop: differential equations}
  ~\begin{enumerate}
  \item The Siegel and the Kudla-Millson theta function are related by
		\begin{align*}
		L_{\frac{3}{2},\tau}\left(\Theta_{KM,D}(\tau,z)\right) =& \frac{1}{4\pi}\Delta_{0,z}\left(\Theta_{S,D}(\tau,z)\right),\qquad
		R_{-\frac{1}{2},\tau}\left(\Theta_{S,D}(\tau,z)\right) = -\pi \Theta_{KM,D}(\tau,z).
		\end{align*}
\item The Shintani and the Millson theta function are related by
		\begin{align*}
		L_{\frac{1}{2},\tau}\left(\Theta_{M,D}(\tau,z)\right) &= \frac{1}{2} L_{2,z}\left(\Theta_{Sh,D}(\tau,z)\right),\qquad
		R_{-\frac{3}{2},\tau}\left(\Theta_{Sh,D}(\tau,z)\right) = \frac{1}{2}R_{0,z}\left( \Theta_{M,D}(\tau,z)\right).
		\end{align*}
  \end{enumerate}
\end{proposition}

\begin{proof}
	These relations can be proved by a direct calculation. See \cite[Theorem 4.2]{bruinierfunke04} for the Siegel and the Kudla-Millson theta function, and \cite[Lemma 3.4]{agor} for the Shintani and the Millson theta function.
\end{proof}

We also need the following expansions of the Kudla-Millson and the Millson theta function.
\begin{proposition}\label{proposition theta functions expansion}
  ~\begin{enumerate}
  \item  The Kudla-Millson theta function has the expansion
\begin{align*}
		\Theta_{KM,D}(\tau,z) &= -\frac{y^{3}}{|D|}\varepsilon(D)\sum_{n=1}^{\infty}n^{2}\left(\frac{D}{n}\right) \\
		& \qquad \times \sum_{\gamma \in \widetilde{\Gamma}_{\infty}\setminus \widetilde{\Gamma}}\left(\exp\left(-\pi\frac{y^{2}n^{2}}{v|D|}\right)v^{-\frac32}\sum_{b \in \Z}e\left(-|D|\frac{b^{2}}{4}\overline{\tau}-bnx \right)\e_{Db}\right)\Bigg|_{\frac32,\widetilde{\varrho}_{L}}\gamma.
	\end{align*}
      \item The Millson theta function has the expansion
		\begin{align*}
		\Theta_{M,D}(\tau,z) &= -\frac{y^{2}}{2i\sqrt{|D|}}\varepsilon(D)\sum_{n=1}^{\infty}n\left(\frac{D}{n}\right) \\
		& \qquad \times \sum_{\gamma \in \widetilde{\Gamma}_{\infty}\setminus \widetilde{\Gamma}}\left(\exp\left(-\pi\frac{y^{2}n^{2}}{v|D|}\right)v^{-\frac12}\sum_{b \in \Z}e\left(-|D|\frac{b^{2}}{4}\overline{\tau}-bnx \right)\e_{Db}\right)\Bigg|_{\frac12,\widetilde{\varrho}_{L}}\gamma.
	\end{align*}
  \end{enumerate}
\end{proposition}
\begin{proof}
	The formula for the Kudla-Millson theta function follows from \cite[Theorem 4.8]{bruinierono}
        together with Proposition \ref{prop: differential equations}, (i).
        The expansion of the Millson theta function is given in \cite[Satz 2.22]{hoevel} (correcting a minor typo).
\end{proof}

\section{The Kudla-Millson and the Siegel theta functions}
\label{sec:KM and Siegel preimage}

Throughout this section, we let $D > 0$ be a fundamental discriminant.
By Proposition~\ref{prop:Lsharpint}, the function
\[
\Theta_{KM,D}^{\#,z}(\tau,z):=\sum_{m \in \Z}\bil{F_{m}-P_{m,y}}{\Theta_{KM,D}(\tau,\cdot)}e(mz)
\]
is a $L_{2,z}$-preimage of $\Theta_{KM,D}(\tau,z)$, which transforms like a modular form of weight two in $z$ and is smooth in $z$. Furthermore, we see that it transforms like a modular form of weight $\frac{3}{2}$ for $\varrho_{L}$ in $\tau$, but its analytic properties as a function of $\tau$ and its growth as $v \to \infty$ are not clear from the definition. We first compute a more explicit formula for $\Theta_{KM,D}^{\#,z}(\tau,z)$.

\begin{theorem}\label{thm:KMD} We have the formula
	\begin{align*}
\Theta_{KM,D}^{\#,z}(\tau,z)=\  2\delta_{D = 1}\mathcal H(\tau)+\widetilde{g}_{0}(\tau,y)
& + 2\sqrt{D}\sum_{m > 0}\sum_{n \mid m}\left( \frac{D}{m/n}\right)n\big(\widetilde{g}_{Dn^{2}}(\tau,my) - g_{Dn^{2}}(\tau)\big)e(mz) \\
& + 2\sqrt{D}\sum_{m < 0}\sum_{n \mid m}\left( \frac{D}{m/n}\right)n\widetilde{g}_{Dn^{2}}(\tau,my)e(mz),
\end{align*}
where $\mathcal{H}(\tau)$ is defined in \eqref{eq:H}, and
\begin{align*}
\widetilde{g}_{0}(\tau,y)& := \frac{1}{4\pi}\sum_{n=1}^{\infty}\left( \frac{D}{n}\right)\sum_{\gamma \in \widetilde{\Gamma}_{\infty}\setminus \widetilde{\Gamma}}\left(v^{-\frac12}\exp\left( -\frac{\pi n^{2}y^{2}}{vD}\right)\e_{0}\right)\bigg|_{\frac32,\varrho_{L}}\gamma, \\
\widetilde{g}_{Dn^{2}}(\tau,y) &:= \frac{1}{4}\sum_{\gamma \in \widetilde{\Gamma}_{\infty}\setminus \widetilde{\Gamma}}\left(\K_{y}\left(Dn^{2}v\right)e\left(-\frac{Dn^{2}\tau}{4}\right)\e_{Dn^{2}}\right) \bigg|_{\frac32,\varrho_{L}}\gamma,
\end{align*}
with the special function
\begin{align*}
\K_{w}(v) &:= \frac{w^{2}}{v^{\frac32}}\int_{1}^{\infty}\exp\left(-\pi \left(\frac{wt}{\sqrt{v}}-\sqrt{v}\right)^{2}\right)tdt.
\end{align*}
\end{theorem}

\begin{remark}
  We have
  \begin{equation*}
    \K_{w}(v) = \frac{\sgn(w)}{2}\erfc\left(\sgn(w)\sqrt{\pi}\left(\frac{ w}{\sqrt{v}}-\sqrt{ v}\right)\right)
                  - \frac{1}{2\pi\sqrt{v}}\exp\left( -\pi\left(\frac{w}{\sqrt{v}}-\sqrt{v} \right)^{2}\right),
  \end{equation*}
  where $\erfc(x) := \frac{2}{\sqrt{\pi}}\int_{x}^{\infty}e^{-t^{2}}dt$ is the \textit{complementary error function}.
\end{remark}
\begin{proof}[Proof of Theorem \ref{thm:KMD}]
	For $m \geq 0$ the Kudla-Millson lift of $F_{m}$ has been computed in \cite[Section~6]{bruinierfunke06} for $D = 1$, and in \cite[Theorem~1.1]{alfesehlen} for $D > 1$. It is given by
	\begin{align*}
	\bil{F_{m}}{\Theta_{KM,D}(\tau,\cdot)} =
	\begin{cases}
	0 & \text{if } m < 0, \\
	2\delta_{D = 1}\mathcal H(\tau) & \text{if } m = 0, \\
	-2\sqrt{D}\sum_{n \mid m}\left( \frac{D}{m/n}\right)ng_{Dn^{2}}(\tau) & \text{if } m > 0.
	\end{cases}
	\end{align*}
	For the Kudla-Millson lift of $P_{m,y}$, we plug in the expansion of $\Theta_{KM,D}(\tau,z)$ given in Proposition~\ref{proposition theta functions expansion}, and use the unfolding argument, to obtain
		\begin{align*}
		\bil{P_{m,y}}{\Theta_{KM,D}(\tau,\cdot)}
		&= -\frac{1}{2D}\sum_{n=1}^{\infty}n^{2}\left(\frac{D}{n}\right)\sum_{\gamma \in \widetilde{\Gamma}_{\infty}\setminus \widetilde{\Gamma}}\int_{y}^{\infty}\int_{0}^{1}e(-mz_{1}) \\
		& \quad \times \left(\exp\left(-\pi\frac{y_{1}^{2}n^{2}}{vD}\right)v^{-\frac32}\sum_{b \in \Z}e\left(-\frac{Db^{2}}{4}\overline{\tau}-bnx_{1} \right) y_{1}^{3} \frac{dx_{1} dy_{1}}{y_{1}^{2}}\e_{Db}\right)\bigg|_{\frac32,\varrho_{L}}\gamma.
		\end{align*}
		First, let $m = 0$. Evaluating the integral over $x_{1}$, we obtain
		\begin{align*}
		\bil{P_{0,y}}{\Theta_{KM,D}(\tau,\cdot)} &= -\frac{1}{2D}\sum_{n=1}^{\infty}n^{2}\left(\frac{D}{n}\right)\sum_{\gamma \in \widetilde{\Gamma}_{\infty}\setminus \widetilde{\Gamma}}\left(v^{-\frac32}\int_{y}^{\infty}\exp\left(-\pi\frac{ y_{1}^{2}n^{2}}{vD}\right) y_{1} dy_{1}\e_{0}\right)\bigg|_{\frac32,\varrho_{L}}\gamma \\
		&= -\frac{1}{4\pi}\sum_{n=1}^{\infty}\left(\frac{D}{n}\right)\sum_{\gamma \in \widetilde{\Gamma}_{\infty}\setminus \widetilde{\Gamma}}\bigg(v^{-\frac12}\exp\left(-\pi\frac{y^{2}n^{2}}{vD} \right)\e_{0}\bigg)\bigg|_{\frac32,\varrho_{L}}\gamma.
		\end{align*}
		Next, for $m \neq 0$, we compute the integral over $x_{1}$ and obtain
		\begin{align*}
		&\bil{P_{m,y}}{\Theta_{KM,D}(\tau,\cdot)}\\
		&= -\frac{1}{2D}\sum_{n \mid m}n^{2}\left(\frac{D}{n}\right)\sum_{\gamma \in \widetilde{\Gamma}_{\infty}\setminus \widetilde{\Gamma}}
		 \quad \bigg(v^{-\frac32}e\left(-D\frac{m^{2}}{4n^{2}}\overline{\tau} \right)\int_{y}^{\infty}\exp\left(2\pi m y_{1}-\pi\frac{\ y_{1}^{2}n^{2}}{vD}\right) y_{1} dy_{1}\e_{D\frac{m}{n}}\bigg)\bigg|_{\frac32,\varrho_{L}}\gamma.
		\end{align*}
		Replacing $n$ by $\frac{|m|}{n}$ yields the stated formulas.
\end{proof}

From Proposition~\ref{prop:Lsharpint} we know that $\Theta_{KM,D}^{\#,z}(\tau,z)$ is a smooth function in $z$ for each fixed $\tau$.
We now show that it is actually smooth in $\tau$ and $z$.

\begin{proposition} \label{prop:KMsharp-smooth}
  The function $\tau\mapsto \Theta_{KM,D}^{\#,z}(\tau,z)$ is smooth and can be differentiated termwise.
\end{proposition}
\begin{proof}
  Recall that the $m$-th Fourier coefficient of $z\mapsto\Theta_{KM,D}^{\#,z}(\tau,z)$ is given by
  \[
    a(\tau; m, y) := 2 \sqrt{D} \sum_{n \mid m} \legendre{D}{m/n} n \left(g_{Dn^2}(\tau)\delta_{n,>0} - \widetilde{g}_{Dn^2}(\tau, my) \right).
  \]
  By Propositions \ref{prop:gest} and \ref{prop:gtest} of the appendix, we have, on any compact $K \subset \uhp \times \uhp$
  for $m > 0$ that
  \[
    a(\tau; m, y) = O_K\left(m^4 \exp\left( \frac{3}{2}\pi m y \right)\right)
  \]
  and similarly for the iterated partial derivatives in $u$ and $v$.
  For $m < 0$ we have that $a(\tau; m, y)$ and all iterated partial derivatives in $u$ and $v$ decay exponentially as $m \to -\infty$.
  This implies that for all $\alpha, \beta \in \N_{0}$, the series
  \[
    \sum_{m}\left(\frac{\partial^\alpha}{\partial u^\alpha}\frac{\partial^\beta}{\partial v^\beta} a(\tau; m, y) \right) e(mz)
  \]
  converges normally on $\uhp$ and thus defines a smooth function in $\tau \in \uhp$ and we have
  \[
   \frac{\partial^\alpha}{\partial u^\alpha}\frac{\partial^\beta}{\partial v^\beta} \sum_{m}a(\tau; m, y) e(mz)
    = \sum_{m} \left(\frac{\partial^\alpha}{\partial u^\alpha}\frac{\partial^\beta}{\partial v^\beta} a(\tau; m, y)\right) e(mz).
  \]
  This finishes the proof.
\end{proof}

We are now ready to prove how $\Theta_{KM,D}^{\#,z}$
behaves under the lowering operators in $\tau$ and $z$,
as well as a differential equation involving the Laplace operators
in both variables.

\begin{proposition}\label{prop: differential equations KudlaMillson theta}
The function $\Theta_{KM,D}^{\#,z}$ satisfies the differential equations
\begin{align}\label{eq Delta ThetaKM}
4\Delta_{\frac{3}{2},\tau}\left( \Theta_{KM,D}^{\#,z}(\tau,z) \right)&= \Delta_{2,z}\left(\Theta_{KM,D}^{\#,z}(\tau,z) \right) ,\\
L_{\frac{3}{2},\tau}\left(\Theta_{KM,D}^{\#,z}(\tau,z)\right) &= -\frac{1}{4\pi}R_{0,z}\left(\Theta_{S,D}(\tau,z)\right), \label{eq L32 ThetaKM}\\
L_{2,z}\left(\Theta_{KM,D}^{\#,z}(\tau,z)\right) &= \Theta_{KM,D}(\tau,z).\label{eq L2 ThetaKM}
\end{align}
\end{proposition}
\begin{proof}
	The formula \eqref{eq L2 ThetaKM} is clear from the construction of $\Theta_{KM,D}^{\#,z}(\tau,z)$.

	We next prove \eqref{eq L32 ThetaKM}. For this, we first write
	\begin{align*}
	L_{\frac{3}{2},\tau}\left(\Theta_{KM,D}^{\#,z}(\tau,z)\right) &= \sum_{m \in \Z}\bil{F_{m}-P_{m,y}}{L_{\frac{3}{2},\tau}\left(\Theta_{KM,D}(\tau,\cdot)\right)}e(mz).
	\end{align*}
        Thus, Proposition \ref{prop:Lsharpint} implies that $L_{\frac{3}{2},\tau}(\Theta_{KM,D}^{\#,z}(\tau,z))$
        is a $L_{2,z}$-preimage of $L_{\frac{3}{2},\tau}\left(\Theta_{KM,D}(\tau,z)\right)$.
	By Proposition~\ref{prop: differential equations} and \eqref{eq laplace via lowering raising}, we have
	\[
	L_{\frac{3}{2},\tau}\left(\Theta_{KM,D}(\tau,z)\right) = \frac{1}{4\pi}\Delta_{0,z} \left(\Theta_{S,D}(\tau,z)\right) = -\frac{1}{4\pi}L_{2,z} \circ R_{0,z}\left(\Theta_{S,D}(\tau,z)\right).
	\]
        Hence, the function $\frac{-1}{4\pi}R_{0,z}(\Theta_{S,D}(\tau,z))$ is also an
        $L_{2,z}$-preimage of $L_{\frac{3}{2},\tau}\left(\Theta_{KM,D}(\tau,z)\right)$.
        Consequently, these two functions can only differ by a weakly homomorphic modular form but since $R_{0,z}(\Theta_{S,D}(\cdot,z)) \in \Amod{2}$
        and $M_2 = \{0\}$, they actually agree.

	Next we show \eqref{eq Delta ThetaKM}. Using Proposition~\ref{prop: differential equations} and \eqref{eq laplace via lowering raising}, \eqref{eq L32 ThetaKM}, and \eqref{eq L2 ThetaKM}, we compute
	\begin{align*}
	\Delta_{2,z}\left(\Theta_{KM,D}^{\#,z}(\tau,z) \right) &= -R_{0,z}\circ L_{2,z}\left(\Theta_{KM,D}^{\#,z}(\tau,z)\right)
	=-R_{0,z}(\Theta_{KM,D}(\tau,z)) \\
	&=\frac{1}{\pi }R_{0,z}\circ R_{-\frac{1}{2},\tau}\left(\Theta_{S,D}(\tau,z)\right)=-4 R_{-\frac{1}{2},\tau}\circ L_{\frac{3}{2},\tau}\left(\Theta_{KM,D}^{\#,z}(\tau,z) \right) \\
	&= 4 \Delta_{\frac{3}{2},\tau}\left( \Theta_{KM,D}^{\#,z}(\tau,z)\right),
	\end{align*}
	completing the proof.
\end{proof}

By Proposition~\ref{prop:Lsharphalfint} we have that
\begin{align*}
(R_{0,z}(\Theta_{S,D}))^{\#,\tau}(\tau,z) &:= \sum_{d \in \Z}\bil{ f_{d}-\mathcal{P}_{d,v}}{R_{0,z}\left(\Theta_{S,D}(\cdot,z)\right)}e\left(\frac{d\tau}{4}\right)\e_{d}
\end{align*}
is a $L_{\frac{3}{2},\tau}$-preimage of $R_{0,z}(\Theta_{S,D}(\tau,z))$. By construction, it transforms like a modular form of weight $\frac{3}{2}$ for $\varrho_{L}$ in $\tau$ and is smooth in $\tau$. Furthermore, itis modular of weight $2$ for $\Gamma$ in $z$.

 We have the following explicit formula for $(R_{0,z}(\Theta_{S,D}))^{\#,\tau}(\tau,z)$.

\begin{theorem}\label{thm:LpreimR0Siegel}
	We have the Fourier expansion
	\[
	(R_{0}(\Theta_{S,D}))^{\#,\tau}(\tau,z) = \sum_{d \in \Z} F_{d,D}^*(z; v)e\left(\frac{d\tau}{4}\right)\e_{d},
	\]
    where
	\begin{equation*}
	F_{d,D}^*(z; v) := F_{d,D}(z)
         + 2\sum_{\lambda \in L_{\frac{dD}{4}} \setminus\{0\}}\chi_{D}(\lambda)\frac{p_{\lambda}(z)}{Q_{\lambda}(z)}
          \exp\left(-\frac{2\pi v R(\lambda,z)}{D}\right),
	\end{equation*}
	with
	\begin{align*}
	F_{d,D}(z) := \begin{dcases}
	-4i \sum\limits_{\substack{Q \in \Gamma \setminus \mathcal{Q}_{-dD} \\ Q > 0}}\frac{\chi_{D}(Q)}{\omega_{Q}}\frac{j'(z)}{j(z)-j(z_{Q})}\quad& \textnormal{if } d > 0,\\
	\frac{2\pi}{3}E_{2}^*(z)\quad& \textnormal{if $d = 0$ and $D = 1$}, \\
	0 \quad& \textnormal{if $d < 0$, or if $d = 0$ and $D > 1$} .
	\end{dcases}
	\end{align*}
	Here the subscript $Q > 0$ indicates that we only sum over positive definite quadratic forms $Q$. If $z_{Q}$ is a CM point then the value $(R_{0}(\Theta_{S,D}))^{\#,\tau}(\tau,z_{Q})$ is obtained by taking the limit $z \to z_{Q}$ in the expression given for $F_{d,D}^*(z; v)$.
	\end{theorem}

\begin{proof}
We first compute the twisted Borcherds lift $\bil{f_{d}}{\Theta_{S,D}(\, \cdot \, ,z)}$ of $f_{d}$. Although it is certainly well-known, we sketch the idea of the computation for the convenience of the reader. By \cite[Theorem 6.1]{bruinierono}, for $z \in \H$ not being a CM point of discriminant $-dD$ in the case $d > 0$, we have the formula
\begin{align*}
\bil{f_{d}}{\Theta_{S,D}(\, \cdot \, ,z)} = \begin{cases} 
-4\log\left|\Psi_{D}(f_{d},z)y^{\frac{c_{f_{d}}(0)}{2}}\right|-c_{f_{d}}(0)(\log(4\pi)+\Gamma'(1)) & \textnormal{if } D = 1, \\
-4\log|\Psi_{D}(f_{d},z)|+2\sqrt{D}c_{f_{d}}(0)L_{D}(1) & \textnormal{if } D > 1, 
\end{cases}
\end{align*}
where $c_{f_{d}}(0)$ denotes the constant coefficient of $f_{d}$ and $\Psi_{D}(f_{d},z)$ is the twisted Borcherds product associated to $f_{d}$, which is defined in \cite[Theorem~6.1]{bruinierono}. Furthermore, $L_{D}(s)$ denotes the usual Dirichlet $L$-function associated to the character of the real quadratic field $\Q(\sqrt{D})$. Recall that $f_{d} = 0$ for $d< 0$, $f_{0} = \theta$ is the Jacobi theta function, and $f_{d} = q^{-d} + O(q)$ for $d > 0$. For $d = 0$ we use the explicit product expansion of the twisted Borcherds product to obtain $\Psi_{1}(\theta,z) = \eta^{2}(z)$ and $\Psi_{D}(\theta,z) = 0$ for $D > 1$. For $d > 0$ we know from \cite[Theorem~6.1]{bruinierono} that $\Psi_{D}(f_{d},z)$ is the unique normalized meromorphic modular form of weight $0$ for $\SL_{2}(\Z)$ having roots or poles of order $\chi_{D}(Q)$ at the CM points $z_{Q}$ for $Q \in \calQ_{-dD}$. This implies that
\[
\Psi_{D}(f_{d},z) = \prod_{\substack{Q \in \Gamma \setminus \mathcal{Q}_{-dD} \\ Q > 0}}(j(z)-j(z_{Q}))^{\frac{\chi_{D}(Q)}{w_{Q}}}.
\]
Altogether, we obtain the formula
\begin{align*}
\bil{f_{d}}{\Theta_{S,D}(\, \cdot \, ,z)} = \begin{cases}
	0\quad& \textnormal{if } d < 0, \\
	-4\log\left|\eta^{2}(z)y^{\frac12}\right|-\log(4\pi)-\Gamma'(1)\quad& \textnormal{if }d = 0 \textnormal{ and } D = 1, \\
	2\sqrt{D}L(1,\chi_{D}) &\textnormal{if }d = 0 \textnormal{ and } D > 1, \\
	-4 \sum\limits_{\substack{Q \in \Gamma \setminus \mathcal{Q}_{-dD} \\ Q > 0}}\frac{\chi_{D}(Q)}{\omega_{Q}}\log\left|j(z) - j(z_{Q})\right| & \textnormal{if } d > 0.
\end{cases}
\end{align*}
For $d =0$ and $D  =1$, we obtain
\[
R_{0}\left( -4\log\left|\eta^{2}(z)y^{\frac12}\right|-\log(4\pi)-\Gamma'(1)\right)  = \frac{2\pi}{3} E_{2}^*(z).
\]
For $d > 0$ and $z$ not a CM point of discriminant $-dD$, we have
\begin{align*}
R_{0}\left(-4 \sum_{Q \in \Gamma \setminus \mathcal{Q}_{-dD}}{\frac{\chi_{D}(Q)}{\omega_{Q}}}\log\left|j(z) - j(z_{Q})\right|\right) &= -4i\sum_{\substack{Q \in \Gamma \setminus \mathcal{Q}_{-dD} \\ Q > 0}}\frac{\chi_{D}(Q)}{\omega_{Q}}\frac{j'(z)}{j(z) - j(z_{Q})}.
\end{align*}

We next assume that $z \in \H$ and suppose that it is not a CM point of discriminant $-dD$. Then we compute, by unfolding,
\begin{align*}
\bil{\mathcal{P}_{d,v}}{\Theta_{S,D}(\, \cdot \,,z)} &= \int_{v}^{\infty}\int_{0}^{1}e\left(-\frac{d\tau_{1}}{4}\right)\Theta_{S,D}(\tau_{1},z)\frac{du_{1}  dv_{1}}{v_{1}^{2}} \\
&= \sum_{\lambda \in L_{\frac{dD}{4}}\setminus \{0\}}\chi_{D}(\lambda)\int_{v}^{\infty} \exp\left(-\frac{2\pi v_{1} R(\lambda,z)}{D}\right)\frac{dv_{1}}{v_{1}}.
\end{align*}
Applying the raising operator $R_{0}$ on both sides, and using the identities
\[
\frac{\partial}{\partial z}R(\lambda,z) = -\frac{i}{2y^2}p_{\lambda}(z)Q_{\lambda}(\overline{z}), \qquad 2R(\lambda,z) = y^{-2}|Q_{\lambda}(z)|^{2},
\]
we obtain
\begin{align*}
\bil{\mathcal{P}_{d,v}}{R_{0,z}\left(\Theta_{S,D}(\, \cdot \,,z)\right)} &= -2\sum_{\lambda \in L_{\frac{dD}{4}}\setminus\{0\}}\chi_{D}(\lambda)\frac{p_{\lambda}(z)}{Q_{\lambda}(z)}\exp\left(-\frac{2\pi v R(\lambda,z)}{D} \right).
\end{align*}
This gives the stated formula for $F^*_{d,D}(z; v)$ if $z$ is not a CM point. On the other hand, it is not hard to see that for fixed $d$ the expression given for $F^*_{d,D}(z; v)$ in the statement of the theorem actually extends to a real analytic function in $z$ on $\H$ if we define its value at a CM point $z_{Q}$ of discriminant $-dD$ by taking the limit $z \to z_{Q}$. The function $\bil{f_{d}-\mathcal{P}_{d,v}}{R_{0,z}\left(\Theta_{S,D}(\, \cdot \, ,z)\right)}$ is real analytic in $z$ on $\H$ since $f_{d} - \mathcal{P}_{d,v}$ is of moderate growth at the cusp $i\infty$. Hence it agrees with
the formula given for $F_{d,D}^*(z; v)$ on all of $\H$. This finishes the proof.
\end{proof}

The following proposition shows that $\Theta_{KM,D}^{\#,z}(\tau,z)$ and $(R_{0,z}(\Theta_{S,D}))^{\#,\tau}(\tau,z)$
agree up to a constant and therefore Theorem \ref{thm:LpreimR0Siegel} in fact gives the Fourier expansion in $\tau$
of $\Theta_{KM,D}^{\#,z}(\tau,z)$.
\begin{proposition} \label{prop:KMDsharp=RSsharp} The identity
\[
  \Theta_{KM,D}^{\#,z}(\tau,z) = -\frac{1}{4\pi}(R_{0,z}(\Theta_{S,D}))^{\#,\tau}(\tau,z)
\]
holds.
\end{proposition}
\begin{proof}
	Fix $z \in \H$. The difference of both sides is modular of weight $\frac{3}{2}$ for $\varrho_{L}$ and holomorphic on $\H$ by Proposition~\ref{prop: differential equations KudlaMillson theta}. The Fourier coefficients (in the expansion with respect to $\tau$) of negative index of the right-hand side vanish as $v \to \infty$ by construction, and it is easy to see that the same holds true for the left-hand side as well. Hence the difference is in $M_{\frac{3}{2},\varrho_{L}} = \{0\}$, which finishes the proof.
\end{proof}

\begin{proof}[Proof of Theorem \ref{introthm:Fdgenser}]
  The theorem in the introduction now follows from Theorem \ref{thm:LpreimR0Siegel} together with
  Proposition \ref{prop:KMDsharp=RSsharp} and Proposition \ref{prop:KMsharp-smooth} by applying the map
  from vector-valued to scalar-valued modular forms given in Remark \ref{rem:vectorscalarvaluedisomorphism}.
\end{proof}

\section{$L_{2,z}$-preimages of the Millson and the Shintani theta functions for $k = 0$}
\label{sec:L2zpreimages Millson and Shintani wt 0}
In this section, we let $-d < 0$ be a negative fundamental discriminant. By Proposition~\ref{prop:Lsharpint} the function
\begin{align*}
\Theta_{M,-d}^{\#,z}(\tau,z) &= \sum_{m \in \Z}\bil{F_{m}-P_{m,y}}{\Theta_{M,-d}(\tau,\cdot)} e(mz)
\end{align*}
is a $L_{2,z}$-preimage of $\Theta_{M,-d}(\tau,z)$. It has weight two in $z$ and weight $\frac{1}{2}$ in $\tau$ for $\overline{\varrho}_{L}$. Further, it is smooth in $z$. Since it can be studied in an analogous way as the function $\Theta_{KM,D}^{\#,z}(\tau,z)$, we leave out some details in the proofs of the following theorem.

\begin{theorem}\label{thm:LpreimMillson}
We have the formula
\begin{align*}
  \Theta_{M,-d}^{\#,z}(\tau,z) = -\widetilde{f}_{0}(\tau,y)
  &-2\sum_{m > 0}\sum_{n \mid m}\left(\frac{-d}{m/n}\right)\left(\sqrt{d}f_{dn^{2}}(\tau) - \widetilde{f}_{dn^{2}}(\tau,my)\right)e(mz) \\
&  -2\sum_{m < 0}\sum_{n \mid m}\left( \frac{-d}{m/n}\right)\widetilde{f}_{dn^{2}}(\tau,my)e(mz),
\end{align*}
where
\begin{align*}
\widetilde{f}_{0}(\tau,y)&:= \frac{1}{4}\sum_{n=1}^{\infty}\left(\frac{-d}{n}\right)\sum_{\gamma \in \widetilde{\Gamma}_{\infty}\setminus \widetilde{\Gamma}}\bigg(\erfc\left( \frac{yn\sqrt{\pi}}{\sqrt{vd}}\right)\e_{0}\bigg)\bigg|_{\frac12,\overline{\varrho}_{L}}\gamma, \\
\widetilde{f}_{dn^{2}}(\tau,y) &:= \frac{1}{4}\sum_{\gamma \in \widetilde{\Gamma}_{\infty}\setminus \widetilde{\Gamma}} \bigg(e\left(-\frac{dn^{2}\tau}{4} \right)\erfc\left(\sqrt{\pi}\sgn(y)\left( \frac{y }{\sqrt{vdn^2}}-\sqrt{vdn^2}\right)\right)\e_{-dn^2}\bigg)\bigg|_{\frac12,\overline{\varrho}_{L}}\gamma.
\end{align*}
\end{theorem}

\begin{proof}
	The Millson lift of $F_{m}$ has been computed in \cite[Theorem 4.6]{agor}, and the Millson lift of $P_{m,y}$ can be computed by unfolding as in the proof of Theorem~\ref{thm:KMD}.
\end{proof}

Analogously to Proposition \ref{prop:KMsharp-smooth}, we obtain that the above representation of $\Theta_{M,-d}^{\#,z}(\tau,z)$ is nicely convergent, so the function is smooth in both variables.

\begin{proposition}
	The function $\Theta_{M,-d}^{\#,z}(\tau,z)$ is smooth in $\tau$, and can be differentiated termwise.
\end{proposition}

The function $\Theta_{M,-d}^{\#,z}(\tau,z)$ also satisfies some interesting differential equations.

\begin{proposition}\label{prop: differential equations Millson theta}
The function $\Theta_{M,-d}^{\#,z}(\tau,z)$ satisfies the differential equations
\begin{align}
4\Delta_{\frac{1}{2},\tau}\left( \Theta_{M,-d}^{\#,z}(\tau,z)\right)&= \Delta_{2,z}\left( \Theta_{M,-d}^{\#,z}(\tau,z)\right) ,  \label{eq Delta ThetaM}\\
L_{\frac{1}{2},\tau}\left(\Theta_{M,-d}^{\#,z}(\tau,z)\right) &= \frac{1}{2}\Theta_{Sh,-d}(\tau,z), \label{eq L12 ThetaM} \\
L_{2,z}\left(\Theta_{M,-d}^{\#,z}(\tau,z)\right) &= \Theta_{M,-d}(\tau,z). \label{eq L2 ThetaM}
\end{align}
\end{proposition}

\begin{proof}
	Using Proposition~\ref{prop: differential equations} we compute
	\begin{align*}
	L_{\frac{1}{2},\tau}\left(\Theta_{M,-d}^{\#,z}(\tau,z)\right)
	&= \frac{1}{2}\sum_{m \in \Z} \bil{F_{m}-P_{m,y}}{L_{2,z}\left(\Theta_{Sh,-d}(\tau,\cdot)\right)} e(mz).
	\end{align*}
	By the same arguments as in the proof of Proposition~\ref{prop: differential equations KudlaMillson theta} we see that the series on the right-hand side equals $\Theta_{Sh,-d}(\tau,z)$.

	The Laplace equation \eqref{eq Delta ThetaM} can be proved using equations \eqref{eq L2 ThetaM} and \eqref{eq L12 ThetaM} and Proposition~\ref{prop: differential equations}.
\end{proof}

By Proposition \ref{prop:Lsharphalfint} the function
\begin{align*}
\Theta_{Sh,-d}^{\#,\tau}(\tau,z) &= \sum_{D \in \Z}\bil{g_{D}-\mathcal{P}_{D,v}}{\Theta_{Sh,-d}(\cdot,z)} e\left(\frac{D\tau}{4}\right)\e_{D}
\end{align*}
is a $L_{\frac{1}{2},\tau}$-preimage of $\Theta_{Sh,-d}(\tau,z)$. It has weight $\frac{1}{2}$ for $\varrho_{L}$ in $\tau$ and weight two for $\Gamma$ in $z$, and it is smooth in $\tau$.

\begin{theorem}\label{thm:LpreimShintani}
	We have the Fourier expansion
	\begin{align*}
	\Theta_{Sh,-d}^{\#,\tau}(\tau,z) = \sum_{D \in \Z} G_{d,D}^*(z; v) e\left( \frac{D\tau}{4}\right)\e_{D},
	\end{align*}
	where
        \begin{equation*}
         G_{d,D}^*(z; v) := G_{d,D}(z) - \frac{\sqrt{d}}{\pi}\sum_{\lambda \in L_{\frac{dD}{4}}\setminus\{0\}}
                            \frac{\chi_{-d}(\lambda)}{Q_{\lambda}(z)}\exp\left(-\frac{2\pi v R(\lambda,z)}{d} \right)
        \end{equation*}
        with
	\begin{align*}
	G_{d,D}(z) := \begin{dcases}
	-\frac{2i}{\pi \sqrt{D}}\sum_{\substack{Q \in \Gamma \setminus \mathcal{Q}_{-dD}\\ Q > 0}}\frac{\chi_{-d}(Q)}{\omega_{Q}}\frac{j'(z)}{j(z)-j(z_{Q})} & \text{if } D > 0,\\
	0 & \text{if } D \leq 0.
	\end{dcases}
	\end{align*}
	If $z_{Q}$ is a CM point then the value $\Theta_{Sh,-d}^{\#,\tau}(\tau,z_{Q})$ is obtained by taking the limit $z \to z_{Q}$ in
        the definition of $G_{d,D}^*(z, v)$.
\end{theorem}

\begin{proof}
	The non-twisted Shimura lift of weakly holomorphic modular forms has been studied in \cite[Example 14.4]{borcherds}. Using similar methods, one can show that the twisted Shimura lift
	\[
	\bil{g_{D}}{\Theta_{Sh,-d}(\tau_{1},z)}
	\]
	of the weakly holomorphic modular form $g_{D}$ for $D > 0$	is a meromorphic modular form of weight two for $\Gamma$ which vanishes at $i\infty$ and has poles at the CM points $z_{Q}$ of discriminant $-dD$ with residues $-\frac{2i}{\pi \sqrt{D}}\chi_{-d}(Q)$. This implies the formula for $G_{d,D}(z)$.

	The inner product involving $\mathcal{P}_{D,v}$ can be computed by unfolding in a similar way as in the proof of Theorem~\ref{thm:LpreimR0Siegel}.
\end{proof}

Finally, we obtain that the above representation of $\Theta_{Sh,-d}^{\#,\tau}(\tau,z)$ in fact gives the Fourier expansion of $\Theta_{M,-d}^{\#,z}(\tau,z)$ in $\tau$.

\begin{proposition}
	The identity
	\[
	\Theta_{M,-d}^{\#,z}(\tau,z) = \frac{1}{2}\Theta_{Sh,-d}^{\#,\tau}(\tau,z)
	\]
	holds.
\end{proposition}
\begin{proof}
	Fix $z \in \H$. The difference of the left-hand side and the right-hand side transforms like a modular form of weight $\frac{1}{2}$ for $\overline{\varrho}_{L}$ and is holomorphic in $\tau$ by Proposition~\ref{prop: differential equations Millson theta}. By construction, the coefficients of negative index in the Fourier expansion (with respect to $\tau$) of $\Theta_{Sh,-d}^{\#,\tau}(\tau,z)$ vanish as $v \to \infty$, and using the explicit formula given in Theorem~\ref{thm:LpreimShintani} we see that its coefficient of index $0$ also vanishes as $v \to \infty$. Further, using Theorem~\ref{thm:LpreimMillson} it is easy to see that $\Theta_{M,-d}^{\#,z}(\tau,z)$ vanishes as $v \to \infty$, as well. Hence the difference of the two functions is a cusp form in $S_{\frac{1}{2},\overline{\varrho}_{L}} = \{0\}$, so they coincide.
\end{proof}

\section{Higher weight}
\label{sec:higherweight}
We discuss two interesting generating functions in higher weights that can be obtained by the method of this paper.
The results can be proved in a similar way as above, so we are brief and leave some of the details to the reader.
As mentioned in the introduction, for the functions considered in this section,
convergence is not an issue and therefore, proving that the generating functions
we consider are smooth in both variables, is much easier than before.
However, the higher weight examples highlight an important feature which
remained hidden in our exposition so far: namely, all the examples that we
consider are in fact a kind of degenerate theta series.
We refer to Remark \ref{rem:phisharp} for some details.

\subsection{A $L_{-k+\frac12,\tau}$-preimage of the Shintani theta function}
\label{sec:heigherweightshintani}

Let $D$ be a fundamental discriminant with $(-1)^{k}D < 0$. We assume that $k > 0$, since the case $k = 0$ is treated in Section \ref{sec:L2zpreimages Millson and Shintani wt 0}. We consider the function
\begin{align*}
\Theta_{Sh,k,D,\perp}(\tau,z) := \Theta_{Sh,k,D}(\tau,z) - \pi_{\text{bil}}(\Theta_{Sh,k,D}(\cdot,z)),
\end{align*}
where
\[
\pi_{\text{bil}}(f(\tau)) := \frac{v^{k+\frac32}}{\Gamma\left(k+\frac12\right)}\sum_{\substack{d < 0 \\ (-1)^{k}d \equiv 0,3 \pmod 4}}(\pi|d|)^{k+\frac12}\bil{ \mathcal{F}_{d}}{f}e\left(\frac{d\bar{\tau}}{4} \right)\e_{d}
\]
denotes the ``holomorphic projection'' of a smooth automorphic form $f$ of weight $-k-\frac32$ for $\widetilde{\rho}_{L}$ with respect to the bilinear pairing\footnote{We have $\pi_{\text{bil}}(f) = v^{k+\frac32}\overline{\pi_{\text{hol}}(v^{-k-\frac32}\overline{f})}$, where $\pi_{\text{hol}}$ denotes the usual holomorphic projection of weight $k+\frac32$ smooth automorphic forms.}. Here
\[
\mathcal{F}_{d}(\tau) := \frac{1}{4}\sum_{\gamma \in \widetilde{\Gamma}_{\infty}\backslash \Mp_{2}(\Z)}\left(e\left(-\frac{d\tau}{4}\right)\e_{d}\right)\bigg|_{k+\frac32,\widetilde{\rho}_{L}}\gamma \in M_{k+\frac32,\widetilde{\rho}_{L}}^{!}
\]
is the usual holomorphic Poincar\'e series of weight $k+\frac32$ for $\widetilde{\rho}_{L}$, which is a cusp form if $d < 0$ and an Eisenstein series if $d = 0$. For $d \in \Z$ we have
\begin{align*}
\bil{\mathcal{F}_{d}}{\Theta_{Sh,k,D}(\tau,z)} = k!\frac{|D|^{\frac{k+1}{2}}}{\pi^{k+1}}f_{k+1,-d,D}(z),
\end{align*}
where
\[
f_{k+1,d,D}(z) := \sum_{\lambda \in L_{-\frac{d|D|}{4}} \setminus \{0\}}\frac{\chi_{D}(\lambda)}{Q_{\lambda}(z)^{k+1}} = \sum_{Q \in \mathcal{Q}_{d|D|}\setminus \{0\}}\frac{\chi_{D}(Q)}{Q(z,1)^{k+1}},
\]
which can easily be shown using the unfolding argument. Note that $f_{k+1,d,D}$ is a cusp form if $d > 0$, an Eisenstein series if $d = 0$, and a meromorphic modular form if $d < 0$.

By construction $\Theta_{Sh,k,D,\perp}(\tau,z)$ is orthogonal to cusp forms of weight $k+\frac32$ with respect to the bilinear pairing. Hence, by Proposition \ref{prop:Lsharphalfint} the function
\begin{align*}
\Theta_{Sh,k,D,\perp}^{\#,\tau}(\tau,z) = \sum_{\substack{d \in \Z \\ (-1)^{k}d \equiv 0,3 \pmod{4}}}\bil{ \mathcal{F}_{d}-\mathcal{P}_{d,v}}{\Theta_{Sh,k,D,\perp}(\cdot,z)}e\left(\frac{d\tau}{4}\right)\e_{d}
\end{align*}
is a $L_{-k+\frac12,\tau}$-preimage of $\Theta_{Sh,k,D,\perp}(\tau,z)$. It has weight $-k+\frac12$ in $\tau$, weight $2k+2$ in $z$, and it is smooth in both variables. By computing the bilinear pairings with $\mathcal{F}_{d}$ and $\mathcal{P}_{d,v}$ explicitly, we obtain a completion of the generating function of the forms $f_{k+1,-d,D}$ for $d \in \Z$. The proof of Theorem \ref{thm:heigherweightintroduction} now follows from a straightforward calculation by unfolding against $\mathcal{F}_{d}$ and $\mathcal{P}_{d,v}$.

\begin{remark}\label{rem:phisharp}
  As we mentioned above, the function $\Theta_{Sh,k,D,\perp}^{\#,\tau}(\tau,z)$ can be written as a theta function, although it is not obtained from a
  Schwartz function.
  However, we have $\Theta_{Sh,k,D,\perp}^{\#,\tau}(\tau,z) = \Theta(\varphi_{Sh,k,D}^{\#,\tau}, \tau, z)$, where
  the ``singular Schwartz function'' $\varphi_{Sh,k,D}^{\#,\tau}(\lambda; \tau, z)$ is given as
  \begin{align*}
       \varphi_{Sh,k,D}^{\#,\tau}(\lambda; \tau, z)
    := \frac{k!\, |D|^{\frac{k+1}{2}}\chi_D(\lambda)}{\pi^{k+1}Q_\lambda(z)^{k+1}} e \left( \frac{Q(\lambda)\tau}{|D|} \right)
    \begin{cases}
      1 - \frac{1}{k!}\Gamma\left(k+1, \frac{2\pi v R(\lambda,z)}{|D|}\right) & \text{ if } Q(\lambda) \geq 0,\\
      \frac{\Gamma\left(k+\frac12,-\pi Q(\lambda)v\right)}{\Gamma\left(k+\frac12\right)} - \frac{1}{k!}\Gamma\left(k+1, \frac{2\pi v R(\lambda,z)}{|D|}\right) & \text{ if } Q(\lambda) < 0.
    \end{cases}
  \end{align*}
It would be interesting to investigate the properties of similar singular Schwartz functions and the associated theta functions, and we hope to come back to this problem in the near future.
\end{remark}

Applying the lowering operator in $z$ to the expression above, we immediately obtain
\[
L_{2k+2,z}\left(\Theta_{Sh,k,D,\perp}^{\#,\tau}(\tau,z)\right) = 2\Theta_{M,k,D}(\tau,z).
\]

To compute the Fourier expansion in $z$,
one can use the Millson theta lift of $F_{m}$,
which can be found in \cite{alfesthesis}, Theorem 4.2.3, and the Millson lifts of the truncated Poincar\'e series can again be computed by unfolding. We leave the details to the reader.

\subsection{A $\xi_{k+\frac32,\tau}$-preimage of the Millson theta function}
Suppose that $D$ is a fundamental discriminant with $(-1)^{k}D < 0$, and assume that $k > 0$ for simplicity. By Proposition \ref{prop:Lsharphalfint} applied to $y^{2k}v^{-k+\frac12}\overline{\Theta_{M,k,D}(\tau,z)}$, the function
\begin{align*}
\sum_{\substack{d \in \Z \\ (-1)^{k}d \equiv 0,3 \pmod{4}}}\left\langle \mathcal{F}_{d}-\mathcal{P}_{d,v},y^{2k}\Theta_{M,k,D}(\cdot,z)\right\rangle e\left(\frac{d\tau}{4}\right)\e_{d}
\end{align*}
is a $\xi_{k+\frac32}$-preimage of $\Theta_{M,k,D}(\tau,z)$. It has weight $k+\frac32$ in $\tau$, weight $-2k$ in $z$, and is smooth in both variables. The inner products against $\mathcal{F}_{d}$ for $d > 0$ have been computed in \cite{bringmannkanekohnen} and yield locally harmonic Maass forms of weight $-2k$ which map to multiples of the cusp forms $f_{k+1,-d,D}(z)$ under $\xi_{-2k}$. The inner products involving $\mathcal{P}_{d,v}$ can easily be computed by unfolding. In this way, we can recover the generating function $\widehat{\Psi}(\tau,z)$ studied in \cite{bringmannkanezwegers}.

\appendix
\section{Growth estimates}
\label{sec:appendix}
\subsection{Weakly holomorphic modular forms of weight $\frac32$}
\label{sec:growth-gD}

For $D > 0$ a discriminant we let $g_{D}$ be the weakly modular form of weight $\frac32$ given in \eqref{eq:gD}, and we use the same symbol for its vector-valued companion. The purpose of this section is to prove the following growth estimate.

\begin{proposition} \label{prop:gest}
  Fix a compact subset $K \subset \H$. For a discriminant $D > 0$ consider the function
\[
  G_D(\tau) = g_D(\tau) - \frac{1}{4}\sum_{\substack{\gamma \in \widetilde{\Gamma}_\infty \backslash \widetilde{\Gamma} \\ D \Im(\gamma\tau) \geq v_0}}\left(e\left( -\frac{D\tau}{4} \right)\e_{D}\right) \bigg\vert_{\frac{3}{2},\varrho_L} \gamma,
\]
where $v_{0} := \min\{\Im(\tau)\, : \, \tau \in K\}$. We have
  \begin{equation*}
   G_{D}(\tau) = O_K(D),
  \end{equation*}
  as $D \to \infty$ through positive discriminants, where the implied constant only depends on $K$. Moreover, the same statement is true for all partial derivatives of $G_D$ of all orders with respect to $u$ and $v$.
\end{proposition}
\begin{proof}
  We prove the claim for $G_D$. The corresponding statement for the partial derivatives follows analogously.

  By \cite[Theorem~3.7]{bruinieronojenkins}, the Fourier expansion of $g_{D} \in M^{!}_{\frac{3}{2},\varrho_{L}}$ is given by
 \[
g_{D}(\tau) = e\left(-\frac{D\tau}{4}\right)\e_{D} - 2\delta_{D = \square}\e_{0} + \sum_{\substack{\nu \in \frac{1}{4}\N\\ -4\nu \equiv 0,1 \pmod{4}} }c_{D}(\nu)e(\nu \tau)\e_{4\nu}
\]
with
\[
c_{D}(\nu) =  24H(4\nu)\delta_{D = \square}+\sum_{\substack{c\in \mathbb{Z}\backslash\{0\}}}\frac{H_{c}\left(-\frac{D}{4},\nu\right)}{\sqrt{2D|c|}}\sinh\left( \frac{4\pi}{|c|}\sqrt{D\nu}\right),
\]
where 
\[
H_{c}(m,n) = e\left(-\frac{3\sgn(c)}{8}\right)\sum_{d\pmod{c}^*}\left\langle \varrho_{L}^{-1}\left(\left(\begin{smallmatrix}
	a & b \\ c & d
	\end{smallmatrix}\right),\sqrt{c\tau + d}\right)\e_{4m},\e_{4n}\right\rangle e\left( \frac{ma + dn}{c}\right)
\]
is a Kloosterman sum. Here $a,b \in \Z$ are chosen such that $\left(\begin{smallmatrix}
	a & b \\ c & d
	\end{smallmatrix}\right) \in \SL_{2}(\Z)$. Furthermore, $\delta_{D = \square}$ equals $1$ or $0$ according to whether $D$ is a square in $\Z$ or not.
 Since the Weil representation factors through a double cover of $\SL_{2}(\Z/2\Z)$, there exists a constant $C > 0$ such that
	\[
	\left|\left\langle \varrho_{L}^{-1}\left(\gamma\right)\e_{4m},\e_{4n}\right\rangle\right| \leq C
	\]
	for all $m,n$ and $\gamma \in \Mp_{2}(\Z)$. 
  The constant term and the Hurwitz class numbers (if $D$ is a square) in the Fourier expansion of $g_{D}$ only contribute $O_K(1)$.

       We split the sum over $c$ in the coefficient $c_{D}(\nu)$ at $\frac{\sqrt{D}}{v_0}$. First we consider the infinite part with $|c| > \frac{\sqrt{D}}{v_0}$,
	\begin{align*}
	\sum_{\nu > 0}\sum_{|c| > \frac{\sqrt{D}}{v_0}}\frac{H_{c}\left(-\frac{D}{4},\nu\right)}{\sqrt{2D|c|}}\sinh\left( \frac{4\pi}{|c|}\sqrt{D\nu}\right)e(\nu \tau)\e_{4\nu}.
	\end{align*}
	Using the series expansion $
	\sinh(z) = \sum_{k=0}^{\infty}\frac{z^{2k+1}}{(2k+1)!}
	$,
	we can write
	\[
	\sinh\left( \frac{4\pi}{|c|}\sqrt{D\nu}\right) = \frac{4\pi}{|c|}\sqrt{D\nu} + O\left( \left(\frac{4\pi}{|c|} \sqrt{D\nu}\right)^{3}e^{4\pi v_0\sqrt{\nu}}\right)
	\]
	for $|c| > \frac{\sqrt{D}}{v_0}$, where the implied constant is independent of $D$, $n$, $\nu$, and $K$. Hence we obtain
	\begin{align*}
	&\sum_{|c| > \frac{\sqrt{D}}{v_0}}\frac{H_{c}\left(-\frac{D}{4},\nu\right)}{\sqrt{2D|c|}}\sinh\left( \frac{4\pi}{|c|}\sqrt{D\nu}\right)= 2\pi\sqrt{2\nu}\sum_{|c| > \frac{\sqrt{D}}{v_0}}\frac{H_{c}\left(-\frac{D}{4},\nu\right)}{|c|^{\frac32}} + O\left(D\nu^{\frac32}e^{4\pi v_0 \sqrt{\nu}}\sum_{|c| > \frac{\sqrt{D}}{v_0}}|c|^{-\frac52}\right),
	\end{align*}
	where we estimate $|H_{c}(-\frac{D}{4},\nu)| \leq C|c|$ in the second sum. Multiplying the second sum by $e(\nu\tau)$ and summing over $\nu$ gives a contribution of $O_K(D)$.
        The first term behaves like
        \[
              2\pi\sqrt{2\nu} \sum_{c\in\mathbb{Z}\backslash\{0\}}\frac{H_{c}\left(-\frac{D}{4},\nu\right)}{|c|^{\frac32}} + O_K\left(\sqrt{D\nu}\right)
        \]
       with an implied constant only depending on $K$.
       The first sum is the Kloosterman zeta function $Z(-\frac{D}{4},\nu, \frac34)$ and Lemma \ref{lem:KloostZest} below shows that this amounts to $O(D \nu^{\frac{3}{2}})$, yielding a total contribution of $O_K(D)$ to the result.

	It remains to bound the part
	\begin{align}\label{eq finite part}
	e\left(-\frac{D\tau}{4}\right)\e_{D} + \sum_{\nu > 0}\sum_{|c| \leq \frac{\sqrt{D}}{v_0}}\frac{H_{c}\left(-\frac{D}{4},\nu\right)}{\sqrt{2D|c|}}\sinh\left( \frac{4\pi}{|c|}\sqrt{D\nu}\right)e(\nu\tau)\e_{4\nu}.
	\end{align}
	To this end, we consider the absolutely convergent sum
	\begin{align}\label{eq finite part 2}
	\frac{1}{4}\sum_{\substack{\gamma \in \widetilde{\Gamma}_{\infty}\setminus \widetilde{\Gamma} \\ |c| \leq \frac{\sqrt{D}}{v_0}}}\left(e\left(-\frac{D\tau}{4}\right)\e_{D}\right)\bigg|_{\frac32,\varrho_{L}}\gamma
	\end{align}
	It is one-periodic and thus has a Fourier expansion, which can be computed in the same way as the expansion of the Maass Poincar\'e series. It turns out that the Fourier expansion equals \eqref{eq finite part}. Furthermore, \eqref{eq finite part 2} differs from
	\[
	\frac{1}{4}\sum_{\substack{\gamma \in \widetilde{\Gamma}_{\infty}\setminus \widetilde{\Gamma} \\ D\Im(\gamma \tau) \geq v_0}}\left(e\left(-\frac{D\tau}{4}\right)\e_{D}\right)\bigg|_{\frac32,L}\gamma
	\]
	(which is the main part of the growth of $g_{D}$ in the theorem) by less than
	\begin{align*}
	& \sum_{|c| \leq \frac{\sqrt{D}}{v_0}}\sum_{\substack{d \in \Z \\ \gcd(c,d) = 1 \\ |d| \geq \sqrt{D} + |cu|}}|c\tau + d|^{-\frac32}e^{\frac{\pi v_0}{2}}
	\ll   \sum_{|c| \leq \frac{\sqrt{D}}{v_0}}\sum_{\substack{d \in \Z}}\left(c^{2} + d^{2}\right)^{-\frac34}e^{\pi\frac{v_0}{2}}  = O_K\left(\sqrt{D}\right).
	\end{align*}
	Up to the estimates for the Kloosterman zeta function given in the following lemma, this finishes the proof.
\end{proof}

The growth behaviour of Kloosterman zeta functions is well-known to experts but somewhat difficult to find in the literature. Therefore we state it in a form which is suitable for our purposes.

\begin{lemma}\label{lem:KloostZest}
  Consider the Kloosterman zeta function
	\[
	Z(m,n;s) := \sum_{c\in\mathbb{Z}\backslash\{0\}}\frac{H_{c}(m,n)}{|c|^{2s}}.
	\]
  It converges to a holomorphic function for $\Re(s)>1$ and has a meromorphic continuation to $\Re(s) > \frac{1}{2}$, with possible poles at the points in $\mathcal{N} := \{ s_n, 1 - s_n \}$, where $\lambda_n := s_n(1-s_n)$ are the eigenvalues of the Laplacian $\Delta_{\frac{1}{2}}$.
  Let $E \subset \{s \in \C:\Re(s)>\frac{1}{2} \}$ be any compact subset. Then, for $s \in E \setminus \mathcal{N}$, we have
  \[
     Z(m,n,s) = O( |mn|),
  \]
  where the implied constant depends only on $E$.
\end{lemma}
\begin{remark}
  We have $\frac{3}{4} \not\in \mathcal{N}$ in Lemma \ref{lem:KloostZest} as there are no exceptional eigenvalues of the hyperbolic Laplacian for the full modular group.
\end{remark}
\begin{proof}[Proof of Lemma \ref{lem:KloostZest}]
  The proof follows from the generalization of Proposition 11 of Appendix E in \cite{hejhal} as described in Section 6 of Appendix E loc. cit.
To obtain exactly the Kloosterman zeta function that we defined above, note that we can write $\varrho_L$ as the product of a representation of $\SL_2(\Z)$ and a multiplier system for weight $\frac{3}{2}$ as follows: Let $\nu$ be the multiplier system on $\SL_2(\Z)$ defined via
\[
  \nu(S) := e^{-\frac{3 \pi i}{4}}, \quad \nu(T) := e^{\frac{3\pi i}{12}}, \quad \nu(\gamma_1 \gamma_2) := \sigma(\gamma_1, \gamma_2) \nu(\gamma_1)\nu(\gamma_2),
\]
where $S:=\left(\begin{smallmatrix}
0 & -1 \\ 1 & 0
\end{smallmatrix}\right)$, $T:=\left(\begin{smallmatrix}
1 & 1 \\ 0 & 1
\end{smallmatrix}\right)$, and where $\sigma(\gamma_1, \gamma_2) = \sigma_{\frac{3}{2}}(\gamma_1, \gamma_2)$ is the usual weight $\frac{3}{2}$-cocycle as defined on page 332 of \cite{hejhal}.
It can be checked that
\[
  \chi(\gamma) := \varrho_L(\widetilde{\gamma}) \nu(\gamma)^{-1}
\]
defines a representation of $\SL_2(\Z)$, where $\widetilde{\gamma} := (\gamma, \sqrt{c\tau+d})$.
This follows from the fact that the cocycle of the Weil representation is equal to $\sigma_{\frac{3}{2}} = \sigma_{\frac{1}{2}}$.
Then the definition of $Z(n,m;s, \mathcal{W})$ with $\mathcal{W}(\gamma) := \chi(\gamma)\nu(\gamma) = \varrho_{L}(\widetilde{\gamma})$ on page 700 of \cite{hejhal} agrees with $Z(m,n;s)$ above.
\end{proof}

\subsection{Non-holomorphic Poincar\'{e} series of weight $\frac32$}

In this subsection, we derive growth estimates for the non-holomorphic Poincar\'e series $\widetilde{g}_{D}(\tau,y)$ defined in Theorem \ref{thm:KMD}.

\begin{proposition} \label{prop:gtest}
	Fix a compact subset $K \subset \H \times \R^{+}$. For a discriminant $D > 0$ consider the function
        \begin{align*}
        \widetilde{G}_{D}(\tau,y) := \widetilde{g}_{D}(\tau,y)
            - \frac{1}{4}\sum_{\substack{\gamma \in \widetilde{\Gamma}_{\infty}\setminus \widetilde{\Gamma} \\ D\Im(\gamma\tau) \geq v_0}}
                                       \left(e\left(-\frac{D\tau}{4} \right) \e_{D}\right)\bigg|_{\frac{3}{2},\varrho_{L}}\gamma,
        \end{align*}
        where $v_{0} := \min\{\Im(\tau) \, : \, (\tau,r) \in K \text{ for some $r \in \R^{+}$}\}$. We have the estimate
	\[
	\sum_{n \mid m}\left(\frac{D}{m/n}\right)n\widetilde{G}_{Dn^{2}}(\tau,my) = O_{K}\left(m^{4}\exp\left(\frac{3}{2}\pi my\right)\right)
	\]
	as $m \to \infty$, where the implied constant depends only $K$ and $D$.
        Moreover, as $m \to -\infty$, we have
	\[
	\widetilde{g}_{D}(\tau,my) = O_{K}\left(m^2 \exp(2 \pi my)\right)
	\]
	with an implied constant only depending on $K$ and $D$. Moreover, the similar estimates hold for all iterated partial derivatives (only the power of $m$ changes, depending on the order of the derivative).
\end{proposition}
\begin{proof}
    We start with $m>0$. Using $\erfc(x) = 2-\erfc(-x)$, we may write
    \begin{align}
    &\widetilde{g}_{Dn^2}(\tau; my)\notag 
    = \frac{1}{4}\sum_{\substack{\gamma \in \widetilde{\Gamma}_\infty \backslash \widetilde{\Gamma} \\ Dn^2\Im(\gamma \tau) \geq v_0}}
          \left( e\left( \frac{-Dn^2\tau}{4} \right) \e_{-Dn^2} \right)\bigg|_{\frac32,\varrho_L} \gamma\ \notag\\
          &\phantom{=}- \
          \frac18\sum_{\substack{\gamma \in \widetilde{\Gamma}_\infty \backslash \widetilde{\Gamma} \\ Dn^2\Im(\gamma \tau) \geq v_0}}
          \left(  e\left( \frac{-Dn^2\tau}{4} \right)\erfc\left( \sqrt{\pi D n^2 v} - \frac{\sqrt{\pi}my}{\sqrt{Dn^2v}} \right) \e_{-Dn^2} \right)\bigg|_{\frac32,\varrho_L} \gamma \label{eq:gtt2}\\
          &\phantom{=} + \frac{1}{8\pi}\!\!\!\sum_{\substack{\gamma \in \widetilde{\Gamma}_\infty \backslash \widetilde{\Gamma} \\ Dn^2\Im(\gamma \tau) \geq v_0}}
          \left( \frac{1}{\sqrt{Dn^2 v}}\exp\left(-\frac{\pi i D n^2u}{2}\right)\exp\left(-\pi\left( \frac{m^2y^2}{Dn^2 v} - 2 m y + \frac{Dn^2 v}{2} \right) \right)\e_{-Dn^2} \right)\bigg|_{\frac32,\varrho_L} \gamma \label{eq:gtt3}\\
          &\phantom{=} +\ \frac14\sum_{\substack{\gamma \in \widetilde{\Gamma}_\infty \backslash \widetilde{\Gamma} \\ Dn^2\Im(\gamma \tau) < v_0}}
          \left( \K_{my}(Dn^2v) e\left( \frac{-Dn^2\tau}{4} \right) \e_{-Dn^2} \right)\bigg|_{\frac32,\varrho_L} \gamma. \label{eq:gtt4}
  \end{align}
     We split \eqref{eq:gtt2} and \eqref{eq:gtt3} into two sums over $Dn^2\Im(\gamma\tau) \geq my$ and $Dn^2\Im(\gamma\tau) < my$, respectively,
     and use the estimate $\erfc(x) \leq \exp(-x^2)$ for $x\geq 0$ and $\erfc(x) \leq 2$ for $x<0$ to obtain a bound for these two sums as follows:
     \begin{multline*}
       \frac{1}{8}\left(1+\frac{1}{\pi \sqrt{v_0}}\right)\exp\left( \frac32 \pi m y \right)
        \sum_{\substack{\gamma \in \widetilde{\Gamma}_\infty \backslash \widetilde{\Gamma} \\ Dn^2\Im(\gamma \tau) \geq \max\{v_0,m_y\}}} \e_{-Dn^2} \Big|_{\frac32,\varrho_L}\gamma \\
        + \frac{1}{4}
        \sum_{\substack{\gamma \in \widetilde{\Gamma}_\infty \backslash \widetilde{\Gamma} \\ my > Dn^2\Im(\gamma \tau) \geq v_0}} \e_{-Dn^2} \Big|_{\frac32,\varrho_L}\gamma
       = O_K\left(m^2 \exp\left( \frac32 \pi m y \right) \right).
     \end{multline*}
For the remaining part \eqref{eq:gtt4}, we note that
\[
  \left|\K_{my}\left(Dn^2v\right) e\left( -\frac{Dn^2 v}{4} \right)\right| = e^{-\frac{\pi Dn^2 v}{2}} \frac{m^2y^2}{(Dn^2v)^{\frac32}}
                                             \int_{1}^\infty \exp\left(2\pi m t - \frac{\pi m^2y^2t^2}{2Dn^2 v}\right)\exp\left(-\frac{\pi m^2y^2t^2}{2Dn^2 v}\right)t dt.
\]
The term $2\pi m t - \frac{\pi m^2y^2t^2}{2Dn^2 v}$ attains a global maximum at $t = \tfrac{2v}{y}$ (which can only happen if $v \geq \tfrac{y}{2}$).
Moreover, since $n \mid m$, we have that $\frac{m^2}{n^2} \geq 1$ and thus we obtain the bound
\[
  \left|\K_{my}\left(Dn^2v\right) e\left( -\frac{Dn^2 v}{4} \right)\right| \leq e^{\frac{3 \pi Dn^2 v}{2}} \frac{m^2y^2}{(Dn^2v)^{\frac32}}
                                             \int_{1}^\infty \exp\left(-\frac{\pi y^2t^2}{2D v}\right)t\, dt = \frac{2}{\pi \sqrt{Dv}} e^{\frac{3 \pi Dn^2 v}{2}-\frac{\pi y^2}{2Dv}}.
\]
Since the series
\[
    \sum_{\gamma \in \widetilde{\Gamma}_{\infty}\setminus \widetilde{\Gamma}}\left(e^{-\frac{\pi y^2}{2Dv}}v^{-\frac{1}{2}}\e_{-Dn^2}\right)\bigg|_{\frac32, \varrho_L} \gamma
\]
is absolutely convergent and for our purposes it is enough to estimate the divisor sum as $O(m^2)$ to finish the proof for $m>0$.
For $m<0$ we can simply use that
\[
  \left|\K_{my}\left(Dn^2v\right) e\left( -\frac{Dn^2 v}{4} \right)\right| \leq \frac{2}{\pi \sqrt{Dv}}\, e^{2\pi my - \frac{\pi y^2}{Dv}}.
\]
With the same bound on the divisor sum, we obtain the statement of the lemma.

To see that the same estimate also holds for the partial derivatives, first note that it is enough to show this for the partial derivatives
with respect to $\tau$ and $\overline{\tau}$ and that
\[
  \frac{\partial}{\partial \tau} = -iR_k +i \frac{k}{v} \ \text{ and }\ \frac{\partial}{\partial \overline{\tau}} = \frac{L_k}{2iv^2}.
\]
To work with $R_k$ and $L_k$ is more convenient since these operators satisfy $L_k (f |_k \gamma) = (L_k f) |_{k-2} \gamma$ and
$R_k (f|_k\gamma) = (R_k f)|_{k+2} \gamma$, respectively.
Now a simple calculation shows that
\[
  L_k \left(\K_{my}\left(Dn^2v\right)\right) = \frac { 4\,\pi \,{D}^{2}{n}^{4}{v}^{2}+2\,\pi \, D {n}^{2}vmy-2\,\pi \,{m}^{2}{y}^{2}+Dn^2v }{ 4(Dn^{2})^{\frac{3}{2}} \sqrt{v}\pi }
{e^{-{\frac {\pi \, \left(Dn^{2}v-my \right) ^{2}}{D {n}^{2}v}}}}
\]
and
\begin{align*}
    & R_k \left(\K_{my}\left(Dn^2v\right)e\left(\frac{-Dn^2\tau}{4}\right)\right)e\left( \frac{Dn^2\tau}{4} \right)\\ &\quad= \frac{3-2\pi v}{4v}\erfc\left(\sqrt{\pi} \left(\frac{my-Dn^2v}{Dn^2v}\right)\right)
 + \frac{2\pi Dn^2v+\pi my+\pi v-1}{2\pi(Dn^2 v)^{\frac{3}{2}}}\exp\left(-\pi \left(\frac{my-Dn^2v}{Dn^2v}\right)^2\right) \\
&\qquad - \frac{m^2y^2}{2(Dn^2v^2)^{\frac52}}\exp\left(-\pi \left(\frac{my-Dn^2v}{Dn^2v}\right)^2\right).
\end{align*}
Hence, we can use similar estimates as the ones given above to obtain the growth estimates for these derivatives
and it is now clear that for higher derivatives, at most the power of $m$ changes.
\end{proof}

\printbibliography
\end{document}